\documentclass[oneside, reqno]{amsart}

\newcommand{\hide}[1]{}

\usepackage{amssymb}
\usepackage{graphicx}
\usepackage{amsmath}

\usepackage{enumitem}
\usepackage{bm}

\usepackage{mathtools}
\usepackage{graphics}
\usepackage{xcolor}
\usepackage{multicol}

\usepackage{textpos}
\usepackage{subfigure}
\usepackage{framed} 
\usepackage{tikz}
\usepackage{tikz-cd}
\usepackage{commath}
\usepackage{nicematrix}
\usepackage{geometry}
\usepackage{float}

\usepackage[colorlinks=true,urlcolor=blue]{hyperref}

\usepackage[doi=false,isbn=false,url=false,eprint=false, style=alphabetic, backend=bibtex,maxbibnames=99]{biblatex}
\usetikzlibrary{shapes, arrows.meta, positioning}

\DeclareMathAlphabet{\mathbbold}{U}{bbold}{m}{n}

\theoremstyle{definition}

\newcommand{\ITM}{\mathrm{ITM}}
\newcommand{\IET}{\mathrm{IET}}

\newcommand{\R}{\mathbb{R}}

\newcommand{\betas}{\beta_*}
\newcommand{\betass}{\beta_{**}}

\renewcommand{\ge}{\geqslant}
\renewcommand{\le}{\leqslant}

\numberwithin{figure}{section}
\numberwithin{equation}{section}

\theoremstyle{plain}
\newtheorem{theorem}{Theorem}[section]
\newtheorem{proposition}[theorem]{Proposition}

\newtheorem*{conjecture*}{Conjecture}

\newtheorem*{mtheorem}{Characterisation of Stability}
\newtheorem*{mconjecture}{Boshernitzan--Kornfeld Conjecture}

\newtheorem*{theorem*}{Theorem}

\newtheorem{lemma}[theorem]{Lemma}
\newtheorem{corollary}[theorem]{Corollary}

\theoremstyle{remark}

\theoremstyle{definition}
\newtheorem{definition}[theorem]{Definition}

\newtheorem{example}[theorem]{Example}

\usepackage{xpatch}
\makeatletter
\AtBeginDocument{\xpatchcmd{\@thm}{\thm@headpunct{.}}{\thm@headpunct{}}{}{}}
\makeatother

\newcounter{lstv}

\newcounter{lsta}

\title[Characterisation of Stability for Interval Translation Maps]{Characterisation of Stability for Interval Translation Maps}
\author[Kostiantyn Drach]{Kostiantyn Drach}
\author[Leon Staresinic]{Leon Staresinic}
\author[Sebastian van Strien]{Sebastian van Strien}

\address{Universitat de Barcelona (Gran Via de les Corts Catalanes, 585, 08007 Barcelona, Spain)}
\address{Centre de Recerca Matem\`atica (Edifici C, Carrer de l'Albareda, 08193 Bellaterra, Spain)}
\email{kostiantyn.drach@ub.edu}

\address{University of Z\"urich (190 Winterthurerstrasse, 8057 Z\"urich, Switzerland)}
\email{leon.staresinic@math.uzh.ch}

\address{Imperial College London (180 Queen's Gate, South Kensington, London SW7 2AZ, UK)}
\email{s.van-strien@imperial.ac.uk}

\thanks{The first author was partially supported from grants CNS2025-166633 (AEI), PID2023-147252NB-I00 (AEI), CEX2020-001084-M (Maria de Maeztu Excellence program), and the ERC Advanced Grant ``SPERIG'' (\#885707). The second author acknowledges the support from Imperial College London through the Roth PhD scholarship and the Swiss National Science foundation through grant number TMCG-2\_213663·2023. The third author acknowledges a partial sponsorship via Bj\"orn Winckler's Marie Curie postdoctoral fellowship \#743959.
\newline 
\indent The content of this paper is part of an earlier manuscript \cite{drach2025densitystableintervaltranslation}, which is available on arXiv. That manuscript has been split into a three-part series: \cite{drach2026transversalityintervaltranslationmaps}, the present paper, and \cite{drach2026topologicalprevalencefinitetype}.}

\begin{document}

\begin{abstract} An \emph{interval translation map} ($\ITM$) is a piecewise translation $T \colon I \to I$ defined on a finite partition $I_1, \ldots, I_r$ of an interval $I$ into $r \ge 2$ subintervals. In contrast to classical interval exchange transformations ($\IET$s), we do not require that the images of these subintervals are disjoint; in particular, $\ITM$s are not assumed to be bijective. Thus, $\ITM$s provide a natural non-invertible generalisation of $\IET$s. 

In this paper, we formulate an appropriate notion of stability for general interval translation mappings and prove a characterisation of stability in terms of two dynamically natural properties called the \textit{Absence of Critical Connections} and \textit{Matching}. This result can be viewed as the foundational step towards the stability theory of general $\ITM$s.
\end{abstract}

\maketitle

\section{Introduction}

\emph{Interval translation maps} ($\ITM$s) are piecewise isometries of an interval that arise as natural non-invertible generalisations of the classical interval exchange transformations ($\IET$s). Throughout the paper, we fix the interval $I := [0,1)$. An $\ITM$ on $r \ge 2$ intervals is specified by

\begin{itemize}
    \item[--] a partition of $I$ by the \emph{discontinuity points} 
    $0 = \beta_0 < \beta_1 < \beta_2 < \dots < \beta_{r-1} < \beta_r = 1$, 
    \item[--] the \emph{translation parameters} $\gamma_1, \dots, \gamma_r$,
\end{itemize}
and it is a map $T \colon I \to I$ that acts by translation on each subinterval of the partition:
\[
T(x) := x + \gamma_i 
\quad \text{for } \quad x \in [\beta_{i-1}, \beta_i),
\quad 1 \le i \le r.
\]
The condition $T(I) \subset I$ forces $\gamma_i \in [-\beta_{i-1},\, 1-\beta_i]$ for all $i$. Consequently, the \emph{parameter space} $\ITM(r)$ of $\ITM$s on $r$ intervals identifies naturally with a convex polytope in $\R^{2r-1}$. Throughout, we endow $\ITM(r)$ with the induced metric.

In contrast with $\IET$s, a typical $T \in \ITM(r)$ is not surjective, i.e.\ $T(I) \subsetneq I$. This lack of surjectivity generates a nested sequence
\[
I \supsetneq T(I) \supseteq T^2(I) \supseteq T^3(I) \supseteq \dots,
\]
which we encode by setting $X_n := T^n(I)$ for $n \ge 0$. Each $X_n$ is a finite union of intervals, and therefore, the intersection
\[
X := \bigcap_{n=0}^{\infty} X_n
\]
is a non-empty \emph{attractor} of the system. This leads to a fundamental dichotomy: an $\ITM$ is said to be of \emph{finite type} if $X_{n+1}=X_n = X$ for some $n$ (and hence for all larger indices), and of \emph{infinite type} if $X_{n+1} \subsetneq X_n$ for all $n$. In fact, the closure of $X$ in $I$ is equal to the \emph{non-wandering set} of $T$ (see Lemma \ref{lem:nonwandering} below). 

The first explicit example of an infinite type map was constructed by Boshernitzan and Kornfeld in~\cite{MR1356616} for $r=3$. They conjectured that such behaviour is rare:

\begin{mconjecture}
\label{conj:inf-type-zero}
For every $r \ge 2$, the set of infinite type $\ITM$s on $r$ intervals 
has Lebesgue measure zero in $\ITM(r)$.
\end{mconjecture}

This conjecture has driven much of the research on $\ITM$s, yet it remains widely open except in a few special cases. Initial progress was obtained for $r=3$ within a special two-parameter family, where the conjecture was proved by Bruin and Troubetzkoy in~\cite{MR2013352}. The Bruin--Troubetzkoy family admits a renormalization scheme analogous to the Rauzy--Veech induction for $\IET$s, and the dynamical properties of this renormalization imply that the set of infinite type maps has zero measure. This parallels the role of the Rauzy--Veech induction in the theory of $\IET$s, which is used to prove that almost every minimal $\IET$ is uniquely ergodic (\cite{MR644018,MR644019}) and weakly mixing (\cite{MR2299743}). Analogous results for typical infinite type maps in the family from ~\cite{MR2013352} were established in ~\cite{MR4973368, artigiani2026typicalweakmixingexceptional} (see also \cite{bruin2023interval}).

A related renormalization approach was later introduced for $r \le 4$ within a broader class of \emph{double rotations}, studied in~\cite{MR2152403, MR2966738,MR4397159}. Ultimately, for $r=3$ the conjecture was established in full in~\cite{MR3124735} by showing that almost every $\ITM$ on three intervals can be renormalized to a double rotation.  Beyond these cases, however, there has been little progress toward a general renormalization theory for $\ITM$s with an arbitrary number of intervals, apart from some special situations (see~\cite{MR2308208}).

The Boshernitzan--Kornfeld conjecture concerns measure-theoretic genericity in $\ITM(r)$. A complementary perspective is provided by topological genericity, where one aims to show that a given property holds on a dense or residual (Baire generic) subset of the parameter space.  A principal tool for establishing such results in dynamical systems is via perturbation theory. 
Broadly speaking, perturbation results allow one to modify the system arbitrarily slightly in order to realise prescribed local dynamical behaviour, while maintaining control over the global dynamics.
By iterating such modifications, one constructs arbitrarily small perturbations that produce global dynamical changes, thereby establishing the genericity of the desired property. Perturbation theory for $\ITM$s was initiated by the authors in \cite{drach2026transversalityintervaltranslationmaps}.

When studying dynamics from a perturbative point of view, a natural problem is to identify properties of a dynamical system that persist under all sufficiently small perturbations. In particular, a dynamical system is called \textit{structurally stable} if its global orbit structure persists under all sufficiently small perturbations. One of the classical aims in the field of dynamical systems is to characterise structurally stable systems in terms of more concrete dynamical properties, and if possible, to prove that such systems are {\lq}typical{\rq}. Let us give several examples of such results.

For $\mathcal{C}^r$ circle diffeomorphisms, with $r \ge 2$, it is known \cite{MR1239171} that the structurally stable systems are exactly the Morse--Smale ones, and that such systems form an open and dense subset of the parameter space. The same holds for $\mathcal{C}^1$ vector fields on two-dimensional manifolds by results of Peixoto \cite{MR101951, MR142859}. For diffeomorphisms, density is impossible in any dimension greater than $1$ by the examples of Newhouse \cite{MR339291}, and the precise characterisation of structural stability is more complicated. It is known that properties Axiom A and strong transversality are sufficient for $\mathcal{C}^r$ structural stability, for any $r \ge 1$, by results of Robbin \cite{MR287580} and Robinson (\cite{MR474411}). The converse is widely believed, but only known for $r=1$ by the result of Ma\~ne \cite{MR932138}. For interval dynamics, $\mathcal{C}^r$-density of structurally stable systems, for any $r \ge 1$, is known by the results of Kozlovski, Shen, and van Strien \cite{MR2335796, MR2342693}, while the corresponding result for $1$-dimensional holomorphic maps was shown by McMullen and Sullivan in \cite{MR1620850}. For exact definitions and results, and a detailed overview of the structural stability for smooth systems, we recommend the survey by Berger \cite{MR3821044}.

In all of the examples above, the main mechanism behind structural stability is the hyperbolicity of maps. The situation for $\ITM$s is drastically different: the derivative is fixed and equal to $1$ everywhere except for finitely many points. Moreover, it turns out that $\ITM$s are never structurally stable or $\Omega$-stable (see Lemma \ref{lem:itm-not-struct-stab} below). For this reason, we define a map to be stable in this context if the next best thing holds: if, loosely speaking, its non-wandering set depends continuously on the system. Let $\overline{X}$ be the closure of $X$ in $[0,1]$; note that the non-wandering set of $T$ is equal to $\overline{X} \setminus \{1\}$ (see Lemma~\ref{lem:nonwandering}).

\begin{definition}[Stable maps]
We say that a map $T \in \ITM(r)$ is \textit{stable} if there a neighbourhood $\mathcal{U}$ of $T$ in $\ITM(r)$ such that:

\begin{enumerate}
\item The mapping assigning to each $\Tilde{T} \in \mathcal{U}$ the corresponding set $\overline{X}(\Tilde{T})$ is continuous with respect to the Hausdorff topology on compact subsets of $[0,1]$. Moreover, for each $\tilde{T} \in \mathcal{U}$, $\overline{X}(\Tilde{T})$ is homeomorphic to $\overline{X}(T)$.
\item The number of discontinuities in $I \setminus \overline{X}(\Tilde{T})$ is constant in $\mathcal{U}$.
\end{enumerate}
\end{definition}

This notion of stability is weaker than the usual ones in $1$-dimensional dynamics because it does not require that the system and its perturbation are conjugate on their non-wandering sets (see \cite{MR1239171, MR1312365}). Note that for $\IET$s, the non-wandering set is always the entire interval, so the main result of this paper becomes trivial.

Since $\ITM$s are never hyperbolic, we need to look for a different mechanism for stability.
Whether or not a map in $\ITM$ is stable turns out to be equivalent to two conditions: the \emph{Absence of Critical Connections} (ACC) and the \emph{Matching} condition. Here ACC means that no iterate of a discontinuity (critical point) is mapped to another discontinuity (see Definition \ref{def:acc}). In the literature on $\IET$s, this is often called the Keane condition. The Matching condition means that the first return map to a non-trivial connected component of the non-wandering set has at most one discontinuity (see Definition \ref{def:matching}). This condition has no clear analogue for invertible systems. It rules out that one connected component of the non-wandering set consists of two pieces with independent dynamics which join up for `coincidental' reasons. This characterisation of stability implies in particular that stable maps are simple: their dynamics on their non-wandering sets corresponds to (a union of) circle rotations. In summary, we have the following theorem, which is the main result of this paper:

\begin{mtheorem}
An interval translation map $T \in \ITM(r)$ is stable if and only if it is of finite type and simultaneously satisfies the Absence of Critical Connections and the Matching conditions.
\end{mtheorem}

It is easy to see that stable maps must be of finite type (Corollary \ref{lem:stable-fin-type}). In general, the mapping $T\mapsto \overline{X}(T)$ is neither upper nor lower semi-continuous, as it is possible for the set $\overline{X}(T)$ to jump up (Example \ref{ex:ghost-preimage}) or down (Figure \ref{fig:pert-a1}) after an arbitrarily small perturbation. In the sequel paper \cite{drach2026topologicalprevalencefinitetype}, using the Characterisation of Stability Theorem, we will show that stable maps are dense in $\ITM(r)$. This will imply that finite type maps contain an open and dense subset of $\ITM(r)$, thus resolving in the positive the topological version of the Boshernitzan--Kornfeld Conjecture.

The structure of the paper is as follows. In Section \ref{sec:preliminaries}, we recall the basic notation, definitions and results for $\ITM$s, and the two main results from \cite{drach2026transversalityintervaltranslationmaps} (the Linear Dependence of Return Map Vectors Lemma~\ref{lem:lin-dep-rj} and the Perturbation Lemma \ref{lem:pert-lem}). In Section \ref{sec:non-wandering}, we show that the set $\overline{X} \setminus \{1\}$ is equal to the non-wandering set of $T$. In Section \ref{sec:stability-overview}, we give an overview of stability, the ACC and the Matching conditions. Finally, in Section \ref{sec:proof-stability=acc+m} we prove the Characterisation of Stability Theorem (stated there as Theorem~\ref{thm:stability=acc+m}).

\section{Preliminaries}
\label{sec:preliminaries}

\subsection{Notation and conventions} 
\label{subsec:not-conv}
In this subsection, we recall some of the main conventions and notation for $\ITM$s. Since the maps we are considering are discontinuous at finitely many points, we will adopt the standard convention of considering every discontinuity $\beta$ of a map $T$ as two points $\beta^- < \beta^+$ and define $T(\beta^-) := \lim_{x \uparrow \beta}$ and $T(\beta^+) := \lim_{x \downarrow \beta}$. We double the preimages of the point $\beta$ in the same way. The images and preimages of $\beta^-, \beta^+$ will be called \textit{signed points}. In a similar way, we can define $x^- < x^+$, $T(x^-)$, $T(x^+)$, etc.\ for any point $x \in I$.

\begin{definition}
\label{def:touching}
We say that a pair of signed points $(a,b)$ \textit{touches} (or \textit{is touching}) if the set $\{a,b\}$ is equal to $\{x^+, x^-\}$ for some point $x \in I$. In that case, we will write $a \sim b$.
\end{definition}
\noindent The \textit{critical set} $\mathcal{C}$ of $T$ is defined as follows:
\[
\mathcal{C} := \{\beta_1^-, \beta_1^+,\dots,\beta_{r-1}^-, \beta_{r-1}^+\}.
\]
We will refer to the elements of $\mathcal{C}$ as either the critical points or the discontinuities, depending on the context. We will sometimes use the labels $\beta_0^+ := 0^+$ and $\beta_r^- := 1^-$, but we do not consider them as critical points. By $\mathcal{C}^+$, we denote the set of all $+$-type points in $\mathcal{C}$, and by $\mathcal{C}^-$ the set of all $-$-type points in $\mathcal{C}$. 

Most of the time, we do not need to know the index of a discontinuity with respect to the order in $I$ nor whether the discontinuity is of $+$-type or $-$-type. That is why we will often use labels $\beta$, $\betas$ and $\betass$ to denote the discontinuities we are dealing with. If we care about the sign of a discontinuity, we will use the labels $\beta^+, \betas^-$, etc. In that case, we will use the same label without a sign to denote discontinuity of $T$ corresponding to the signed discontinuity, e.g.\ we denote by $\beta$ the discontinuity such that $\beta^+$ ($\beta^-$) is the $+$-part ($-$-part) of $\beta$.

The perturbation of (or a map sufficiently close to) some starting map $T$ will be denoted by $\Tilde{T}$. Many of the most important objects associated to a map $T$, e.g.\ critical points, the set $X$ and boundary points of intervals contained in $X$, will have well-defined continuations for sufficiently small perturbations of $T$. For any such object $Z$, we will denote its continuation by $\Tilde{Z}$, e.g.\ $\Tilde{\beta}^+$, $\Tilde{X}$, $\Tilde{J}$ and similar.

We associate to each point $x \in I$ its \textit{itinerary}. The itinerary of $x$ is an infinite sequence of integers $(i_0(x), i_1(x), \dots, i_n(x), \dots)$, with $1 \le i_n(x) \le r$, where $i_n(x) = s$ means that $T^n(x) \in I_s$. We will often use `itinerary up to time $n$' to refer to the first $n$ elements of this sequence. 

Finally, for a point $x$ we will refer to the set $O(x) := \{ x, T(x), T^2(x), \dots \}$ as the \textit{$T$-orbit} of $x$. Moreover, if $S$ is a subset of $I$, we will refer to the set $O(S) := \bigcup_{x \in S} O(x)$ as the $T$-orbit of $S$. If the map $T$ is clear from the context, we will omit it and simply use the term `orbit'. We will refer to the set $O(x,n) := \{ x, \dots, T^{n-1}{x} \}$, where $n \ge 0$, as the orbit up to time $n$ of $x$, and to the set $O(S,n) := \bigcup_{x \in S} O(x,n)$ as the orbit up to time $n$ of $S$.

\subsection{Basic results and definitions}

\label{subsec:basic-res-def}
In this subsection, we recall the main definitions and elementary facts for $\ITM$s. For proofs, we refer to Section $2$ of \cite{drach2026transversalityintervaltranslationmaps}. We start with the following useful definition:

\begin{definition}
\label{def:c1c2}
For $T \in \ITM(r)$, we define the following sets:
\begin{itemize}
    \item $\mathcal{C}_1$ is the set of all $\beta \in \mathcal{C}$ that eventually land on a discontinuity;
    \item $\mathcal{C}_2$ is the set of all $\beta \in \mathcal{C}$ that never land on a discontinuity, but are eventually periodic;
    \item $\mathcal{C}_0 = \mathcal{C}_1 \cup \mathcal{C}_2$;
    \item $\mathcal{C}_i^{\pm} = \mathcal{C}^{\pm} \cap \mathcal{C}_i$, for $i \in \{ 0, 1, 2 \}$.
\end{itemize}
\end{definition}

Periodic points of $\ITM$s form entire intervals of periodic points that are called maximal periodic intervals:

\begin{definition}
\label{def:max-per-int}
A maximal interval $J \subset I$ consisting of periodic points with the same itinerary is called a \textit{maximal periodic interval}. It is clear that $J$ is a half-open interval and that its boundary points have the property that they land on discontinuities of $T$ or on the boundary points of $I$. Moreover, no point in the interior of $J$ lands on a discontinuity. 
\end{definition}

Since there are $r-1$ discontinuities for a map $T \in \ITM(r)$, it is clear that there can only be finitely many maximal periodic intervals with pairwise disjoint orbits.

The second type of intervals we are interested in are the connected components of $X$ that are equal to an interval. It will be convenient to use the following shorthand for such intervals:

\begin{definition}
\label{def:comp-int}
We will say that an interval $J$ is an \textit{interval component} of $X$ if $J$ is equal to a connected component of $X$.
\end{definition} 

Note that an interval component of $X$ can be a maximal periodic interval, but that a maximal periodic interval can be strictly contained in an interval component of $X$.

We now recall several elementary results about general $\ITM$s. The following is a fundamental classification of the dynamics of every point in $I$:

\begin{lemma}[Orbit Classification Lemma]
\label{lem:orb-class}
For each point $x \in I$, at least one of the following three possibilities holds:
\begin{enumerate}
\item (Precritical) $x$ lands on a discontinuity of $T$;
\item (Preperiodic) $x$ lands on a periodic point of $T$;
\item (Accumulation) $x$ accumulates on a discontinuity of $T$. More precisely, there exist discontinuities $\beta, \betas \in \mathcal{C}$ such that $x$ accumulates on $\beta$ from the left and on $\betas$ from the right. \qed
\end{enumerate} 
\end{lemma}
Recall the definition of a first return map:
\begin{definition}[First return map]
\label{def:rj}
For an interval $J \subset X$, define $R_J$ to be the \emph{first return map} (or simply the \emph{return map}) to $J$ under $T$, i.e.\ for every $x \in J$ that returns to $J$, we define $R_J(x) := T^k(x)$ for the smallest integer $k = k(x) \ge 1$ such that $T^k(x) \in J$.
\end{definition}
The following is a simple lemma about the first return maps for intervals in $X$:
\begin{lemma}
\label{lem:rj-facts}
The domain of $R_J$ is the entire interval $J$. $R_J: J \to J$ is bijective and $J$ is partitioned into finitely many maximal half-open subintervals such that no point in their interiors lands on a discontinuity of $T$. \qed
\end{lemma}
The return map $R_J$ is continuous on these subintervals of $J$, so for simplicity we refer to these intervals as `continuity intervals of $R_J$'. We usually use $x$ and $y$ to denote the boundary points of $J$, so that $J = [x,y)$.

We use the following notation for such interval components $J$. By Lemma \ref{lem:rj-facts}, there are finitely many points $a_{1}, \dots, a_{N-1}$ in the interior of $J$ that land on discontinuities before returning to $J$. It is also convenient to define $a_{0} := x$ and $a_{N} := y$. The continuity intervals of $R_J$ are denoted by $J_1, \dots, J_N$. The return time of $J_j$ to $J$ is denoted by $r_j$, for $1 \le j \le N$, and the landing times of $a_j$ to discontinuities of $T$ are denoted by $l_j$, for $1 \le j \le N-1$.
Finally, Lemma \ref{lem:J-dynamics} gives some elementary dynamical properties for the interval components of $X$. This result will be implicitly used throughout the paper.
\begin{lemma}
\label{lem:J-dynamics}  
Let $T \in \ITM(r)$. Then the following two properties hold: 
\begin{enumerate}[label=(\alph*)]
    \item Every interval component of $X$ is of the form $[T^{k_1}(\beta^+), T^{k_2}(\betas^-))$ for some $\beta^+, \betas^- \in X \cap \mathcal{C}$ and $k_1,k_2 \ge 0$;
    \item If $T^{l_1}(\beta)$ is an interior point of an interval component of $X$ for some $\beta \in X \cap \mathcal{C}$ and $l_1 \ge 0$, then there exist $\betas \in X \cap \mathcal{C}$ and $l_2 \ge 0$ such that $T^{l_1}(\beta) \sim T^{l_2}(\betas)$. \qed
\end{enumerate}
\end{lemma}

\subsection{Dynamically defined vectors and the Perturbation Lemma}
\label{subsec:vec+pert-lemma}

In this subsection, we recall the product notation, the dynamically defined vectors and the two main results of \cite{drach2026transversalityintervaltranslationmaps}, the Linear Independence of Return Map Vectors \ref{lem:lin-dep-rj} and the Perturbation Lemma \ref{lem:pert-lem} that will be used in the proof of Theorem \ref{thm:stability=acc+m}.

For each  $T$,  $x \in I$, $s=1,\dots,r$ and $n\ge 1$, let:
\[
k_s(x,n,T) := \# \{ T^j(x) \in I_s; 0\le j < n \}.
\]
Thus $k_s(x,n,T)$ represents the number of entries of $x$ into $I_s$ up to time $n$. When the map $T$ is clear from context, we simply write $k_s(x,n)$.

Let $W(r) := \R{}^{r} \oplus \R{}^{r-1}$ be the $(2r-1)$-dimensional real vector space that contains the \textit{coefficient vectors} (introduced below), and let $(\bm{e}_1, \dots, \bm{e}_r)$ and $(\bm{f}_1, \dots, \bm{f}_{r-1})$ be the canonical bases for $\R{}^{r}$ and $\R{}^{r-1}$, respectively. Recall that $\ITM(r)$ is the parameter space of $\ITM$s on $r$ intervals, and that it is a convex polytope contained in $\mathbb{R}^{2r-1}$. We call the elements of this space \textit{parameter vectors}. These vectors also have canonical coordinates coming from the ambient space $\mathbb{R}^{2r-1}$. We will use the shorthand $(\gamma \, \beta)$ for a parameter vector $(\gamma_1 \dots \gamma_r \, \beta_1 \dots \beta_{r-1})$.

Let $\langle \cdot, \cdot \rangle$ be the standard scalar product on $\mathbb{R}^{2r-1}$. Since $W(r) = \R{}^{r} \oplus \R{}^{r-1}$ and $\ITM(r)$ is a subset of $\mathbb{R}^{2r-1}$, it makes sense to write $\langle v, (\gamma \, \beta) \rangle =  \sum_{s=1}^r v_s \gamma_s + \sum_{s=1}^{r-1} v_{s+r} \beta_s$ for a coefficient vector $v = \sum_{s=1}^r v_s \bm{e}_s + \sum_{s=1}^{r-1} v_{s+r} \bm{f}_s \in W(r)$ and a parameter vector $(\gamma \, \beta) \in \ITM(r)$. We call $\langle v, (\gamma \, \beta) \rangle$ the \textit{product} of $v$ and $(\gamma \, \beta)$.

Let $J = [x,y)$ be an interval component (Definition \ref{def:comp-int}) of $X$. The first return map $R_J$ to $J$ is well-defined (Lemma \ref{lem:rj-facts}) and there are finitely many points $a_{1}, \dots, a_{N-1}$ in the interior of $J$ that land on discontinuities before returning to $J$. Let $J_1, \dots, J_{N}$ be the continuity intervals of $R_J$, and let $r_j$ be the return time of $J_j$ to $J$. Let $a_{0}^{+} := x^+$ and $a_{k}^{-} := y^-$ be the boundary points of $J$. For each $1 \le j < N$, let $m_{j}^{+}$ be the number of discontinuities that $a_{j}^{+}$ lands on before returning to $J$ and, for $1 \le k \le m_{j}^{+}$, let $\beta^{+}(j,k)$ be the $k$-th discontinuity along the orbit up to return time to $J$ of $a_j^{+}$. Define $m_{j}^{-}$ and $\beta^{-}(j,k)$ analogously, for $1 \le j < N$ and $1 \le k \le m_{j}^{-}$. For each $1 \le j < N$, the points $\beta^{+}(j,1)$ and $\beta^{-}(j,1)$ are the $+$ and $-$ part of a single discontinuity, which we denote by $\beta(j)$.

In what follows, it will be useful to denote by $\text{ind}(\beta) \in \{1, \dots, r-1\}$ the index of a discontinuity $\beta$ with respect to this order inside $I$.

We now define three types of dynamically defined vectors associated to the orbit of $J$. The first ones are the \textit{first landing vectors}, that correspond to the first time the points $a_1^J, a_2^J, \dots, a_{N_J-1}^J$ land on discontinuities of $T$:

\begin{definition}[First landing vectors]
\label{def:lan-vec}
For $1 \le j < N$, let $l_{j}$ be the landing time of $a_{j}$ to $\beta(j)$ and let $L_{j} \in W(r)$ be the associated vector, called the \emph{first landing vector}:

\[
L_{j} \coloneqq \left(\sum_{s = 1}^r k_s(a_j, l_{j}) \, \bm{e}_s, \, - \bm{f}_{\text{ind}(\beta(j))}\right).
\]
\end{definition}

If $a_j$ is a discontinuity of $T$, i.e.\ if the landing time $l_j$ is zero, then $L_j = \left( 0, - \bm{f}_{\text{ind}(\beta(j))}\right)$ by definition. We call $L_j$ the first landing vector of $a_j$ to $\beta(j)$ because the following holds:

\begin{align*}
0 &= a_j + \sum_{s = 1}^r k_s(a_j, l_{j}) \gamma_s - \beta(j) \\
&= \left(\sum_{s = 1}^r k_s(a_j, l_{j}) \gamma_s - \beta(j)\right) + a_j \\
&= \langle L_j, (\gamma \, \beta) \rangle + a_j.
\end{align*}

The second type of vectors we need to consider are the \textit{critical connection vectors}, which correspond to landings of discontinuities in the orbit of $J$ onto other discontinuities (which are thus also in the orbit of $J$).

\begin{definition}[Critical connection vectors]
\label{def:cc-vec}
For each $1 \le j < N$ and $1 \le k < m_{j}^{+}$, let $q^{+}(j,k)$ be the landing time of $\beta^{+}(j,k)$ to $\beta^{+}(j,k+1)$ and let $C^{+}(j,k) \in W(r)$ be the associated vector, called the \emph{critical connection vector}:

\[
C^{+}(j,k) \coloneqq \left(\sum_{s=1}^r k_s(\beta^{+}(j,k), q^{+}(j,k)) \, \bm{e}_s, \, \bm{f}_{\text{ind}(\beta^{+}(j,k))} - \bm{f}_{\text{ind}(\beta^{+}(j,k+1))}\right).
\]
\end{definition}%
The following holds by construction:

\begin{align*}
0 &= \beta^{+}(j,k) + \sum_{s=1}^r k_s(\beta^{+}(j,k), q^{+}(j,k)) \gamma_s - \beta^{+}(j,k+1) \\
&= \sum_{s=1}^r k_s(\beta^{+}(j,k), q^{+}(j,k)) \gamma_s + \beta^{+}(j,k) - \beta^{+}(j,k+1) \\
&= \langle C^{+}(j,k), (\gamma \, \beta) \rangle.
\end{align*}
Define $C^{-}(j,k) \in W(r)$ analogously for $1 \le j < N$ and $1 \le k < m_{j}^{-}$.

The third type of dynamical vector we need to consider is the \textit{return vectors}, which correspond to the return of points $a_1^{-}, a_1^{+}, \dots, a_{N-1}^{+}, a_{N-1}^{-}$ to the interval $J$.

\begin{definition}[Return vectors]
\label{def:ret-vec}
For each $1 \le j < N$, let $r_{j}^{+}$ be the time at which $\beta^+(j,m_{j}^{+})$ lands into $J$ and let $R_{j}^{+} \in W(r)$ be the associated vector, called the \emph{return vector}:  

\[
R_{j}^{+} \coloneqq \left(\sum_{s=1}^r k_s(\beta^{+}(j,m_{j}^{+}), r_{j}^{+}) \, \bm{e}_s, \, \bm{f}_{\text{ind}(\beta^{+}(j,m_{j}^{+}))}\right).
\]
\end{definition}
By definition, the following holds:
\[
\langle R_j^{+}, (\gamma \, \beta) \rangle = \sum_{s=1}^r k_s(\beta^{+}(j,m_{j}^{+}), r_{j}^{+}) \gamma_s + \beta^+(j,m_{j}^{+}) \in J.
\]
Note that also:
\begin{equation}
\label{eq:R-vec-property}
R_J(a_{j}^+) = \langle R_j^{+}, (\gamma \, \beta) \rangle 
\end{equation}
since $a_{j}$ lands on $\beta^+(j,m_{j}^{+})$. Define $r_{j}^{-}$ and $R_{j}^{-} \in W(r)$ analogously for $1 \le j < N$.

For the boundary points $x^{+} = a_0^{+}$ and $y^- = a^{-}_{N}$, the definitions of the corresponding dynamical vectors depend on whether they land on discontinuities of $T$ before returning to $J$ or not. If $a_0^{+}$ lands on a discontinuity before returning to $J$, then we may analogously as for other $a_j^{+}$, where $1 \le j < N$, define the vectors $L_0$, $C^{+}(0,k)$ and $R^{+}_0$. In the case when $a_0^{+}$ does not land on a discontinuity of $T$ before returning to $J$, we only define the return vector $R_0^{+}$ to $J$:

\[
R_0^{+} := \left(\sum_{s=1}^r k_s(a_0^{+}, r_{0}^{+}) \, \bm{e}_s, \, 0 \right),
\]
where $r_0^+$ is the return time of $a_0^+$ to $J$. Note that in this case:
\[
\langle R_0^{+}, (\gamma \, \beta) \rangle = R_J(a_{0}^+) - a_0.
\]
Analogously, if $a_{N}^{-}$ lands on a discontinuity of $T$, we may define the vectors $L_{N}, C^{-}(N,k)$ and $R^{-}_{N}$. If $a_{N}^{-}$ does not land on a discontinuity of $T$ before returning to $J$, then define $R_{N}^{-}$ analogously as for $a_0^{+}$.

Because the definitions of the dynamical vectors associated to the boundary points of the interval $J$ are different depending on whether they land on discontinuities of $T$ or not, this leads to different statements of Lemma \ref{lem:lin-dep-rj}, depending on whether these landings happen or not. For simplicity, we will assume that these boundary points land on discontinuities of $T$, because this case is more complicated, and we state Lemma \ref{lem:lin-dep-rj} with this assumption.

\begin{lemma}[Linear Dependence of Return Map Vectors]
\label{lem:lin-dep-rj}
Let $T \in \ITM(r)$, and let $J$ be an interval component of $X$. Assume that there exist real coefficients $\alpha_{j}, \alpha_j^{+}, \alpha_j^{-}, \alpha^{+}(j,k), \alpha^{-}(j,k)$ such that:
\begin{align}
\label{eq:lin-dep-sum}
\begin{split}
&\sum_{j=0}^{N-1} \left( \sum_{k=1}^{m_j^{+}-1} \alpha^{+}(j,k) C^{+}(j,k) + \alpha_j^{+} R_j^{+} \right) \\ 
+&\sum_{j=1}^{N} \left( \sum_{k=1}^{m_j^{-}-1} \alpha^{-}(j,k) C^{-}(j,k) + \alpha_j^{-} R_j^{-} \right) \\
+& \left( \sum_{j=0}^{N} \alpha_j L_j \right) = 0.    
\end{split}
\end{align}
Then the following equalities hold:
\begin{align}
\label{eq:lin-dep-equality1}
\begin{split}
&\alpha^{+}(j,1) = \dots = \alpha^{+}(j,m_j^{+}-1) = \alpha_j^{+} \\
&=-\alpha^{-}(j+1,1) = \dots = - \alpha^{-}(j+1,m_{j+1}^{-}-1) = -\alpha_{j+1}^{-},
\end{split}
\end{align}
for all $0 \le j < N$. Moreover, $\alpha_j = \alpha^-(j,1) + \alpha^+(j,1)$ for all $1 \le j < N$.
\end{lemma}

In the case when the boundary point $a_0^+$ of $J$ does not land on a discontinuity of $T$, we need to remove $L_0$ and $C^+(0,k)$ from \eqref{eq:lin-dep-sum} and all of the coefficients for $j=0$ from the first line of \eqref{eq:lin-dep-equality1} except $\alpha_0^+$. A similar procedure should be done when $a_N^-$ does not land on a discontinuity.

To state the Perturbation Lemma \ref{lem:pert-lem}, we need the following useful definition:

\begin{definition}
\label{def:trans-fac}
Let $n > 0$ be a natural number. We define the \textit{translation factor} $Tr(x,n)$ at time $n$ of a point $x \in I$ as the translation factor of $T^n$ restricted to $x$, i.e.\ $Tr(x,n) = T^n(x)-x$.
\end{definition}

This definition can easily be extended to intervals $K \subset I$ of points that have the same itinerary up to time $n > 0$, by calling the translation factor $Tr(K,n)$ at time $n$ of $K$ the translation factor of any point $x$ in $K$. Note that $Tr(K,n) = 0$ if and only if $K$ is periodic with period $n$. We denote the translation factors defined for a perturbation $\Tilde{T}$ of $T$ by $\Tilde{T}r$.

\begin{lemma}[Perturbation Lemma]
\label{lem:pert-lem}
Let $T$ be an interval translation map such that $T(I)$ is compactly contained in the interior of $I$. Let $J$ be one of the following: an interval component of $X$ or a maximal periodic interval. Then there exists an $\epsilon_0 > 0$ depending on $J$ and $T$, with the following properties. For every $\epsilon < \epsilon_0$ and every choice of $\epsilon^{\gamma}_1, \dots \epsilon^{\gamma}_{N}, \epsilon^{\beta}_0, \dots \epsilon^{\beta}_{N} \in (-\epsilon, \epsilon)$ there exists a perturbation $\Tilde{T}$ such that $|\Tilde{T} - T| \to 0$ as $\epsilon \to 0$, with $|\cdot|$ being the distance in $\ITM(r)$, and the following holds:

\begin{enumerate}
    \item There exists an interval $\Tilde{J} \subset I$ that is $\epsilon$-close to $J$ and partitioned into intervals $\Tilde{J}_j = [\Tilde{a}_{j-1},\Tilde{a}_j)$, with $1 \le j \le N$, such that $\Tilde{J}_j$ maps forward continuously up to time $r_j$ under the iterates of $\Tilde{T}$ and has the same itinerary up to time $r_j$ as $J_j$ for all $1 \le j \le N$. In the case when $J$ is an interval component of $X$, we may set $\Tilde{a_j} - a_j = \epsilon^{\beta}_j$ for all $0 \le j \le N$;
    \item The difference between the translation factors of $\Tilde{J}_j$ and $J_j$ is $\epsilon^{\gamma}_j$ for all $1 \le j \le N$, i.e.\ $\Tilde{T}r(\Tilde{J}_j,r_j) - Tr(J_j,r_j) = \epsilon^{\gamma}_j$.
\end{enumerate}
Moreover, we have that:

\begin{enumerate}[label=(\alph*)]
    \item Let $T^n(\beta^+) = \betas^+$ be a critical connection such that $\beta^+$ and $T^n(\beta^+)$ are both contained in the $T$-orbit of $J_j$ up to time $r_j$ for some $1 \le j \le N$. We may assume that $\Tilde{\beta}^+$ is still contained in the $\Tilde{T}$-orbit of $\Tilde{J}_j$ up to time $r_j$ and the difference $\Tilde{T}^n(\Tilde{\beta}^+) - \Tilde{\beta}_*^+$ can be chosen arbitrary in $[0,\epsilon)$. Similarly, for critical connections $T^n(\beta^-) = \betas^-$ and the difference $ \Tilde{\beta}_*^- - \Tilde{T}^n(\Tilde{\beta}^-)$;
    \item For every critical connection $T^n(\beta) = \betas$, with $\beta,\betas \notin O(J)$, such that either $\beta \notin X$ or $\beta, \betas$ are part of a single periodic orbit, the difference $\tilde T^n(\tilde \beta) - \tilde \betas$ can be chosen arbitrary in $(-\epsilon,\epsilon)$. \qed
\end{enumerate}
\end{lemma}

In (a), the changes are such that the itineraries of the critical points are preserved, while this is not necessarily the case in part (b). The assumption that $T(I)$ is compactly contained in the interior of $I$ holds for every $\ITM$ in a complement of finitely many hyperplanes contained in the boundary of $\ITM(r)$, so this may be assumed without loss of generality, and it allows us not to include several necessary assumptions on the $\epsilon, \epsilon^{\gamma}_1, \dots \epsilon^{\gamma}_{N}, \epsilon^{\beta}_0, \dots \epsilon^{\beta}_{N}$ that guarantee that the perturbed map still has image contained in $[0,1)$.

\section{The non-wandering set}
\label{sec:non-wandering}

As usual, define the non-wandering set $\Omega(T)$ of $T$ to be the set of $x \in I$ so that for each neighbourhood $U$ of $x$ there exists $n>0$ so that $T^n(U)\cap U\ne \emptyset$. The set $X = \bigcap_{n=1}^{\infty} X_n$ is by definition the \textit{attractor} of $T$. In special cases, it is possible for $X$ to accumulate on $1$, which is not contained in $I=[0,1)$. Because of this, we assume the closure $\overline{X}$ of $X$ is defined as the closure of $X$ in $[0,1]$, and not $I$, so that $\overline{X}$ is compact in $[0,1]$. Moreover, we also consider other topological notions as happening in $[0,1]$, e.g.\ the interior of a set or a limit of a sequence.

Our main result of this subsection is that the set $\overline{X} \setminus \{1\}$ is equal to the non-wandering set of $T$, Lemma \ref{lem:nonwandering}. We will use the following lemma, proved in \cite{drach2026transversalityintervaltranslationmaps}, that completely describes the topological structure of the closure (in $[0,1]$) of $X$.

\begin{lemma}
\label{lem:x-structure}
For any map $T$, $\overline{X}$ is equal to $A_1 \cup A_2$, where $A_1$ is a finite union of closed intervals and $A_2$ is a Cantor set. The union is disjoint, except possibly at the right endpoints of the intervals in $A_1$. Moreover, one of the sets $A_1, A_2$ is allowed to be empty. \qed
\end{lemma}

Recall that all of the intervals we consider are half-open and of the form $[a,b)$, unless stated otherwise.

\begin{lemma}
\label{lem:nonwandering} 
The set $\overline{X} \setminus \{ 1 \}$ is equal to the non-wandering set $\Omega(T)$.
\end{lemma}

\begin{proof}
We first prove that $\overline{X} \setminus \{1\} \supseteq \Omega(T)$. Let $x$ be a point in $I$ such that $x \notin \overline{X}$. This means that $\overline{\mathcal{U}} \cap \overline{X} = \emptyset$ for any sufficiently small neighbourhood $\mathcal{U}$ of $x$. Take one such sufficiently small open neighbourhood $\mathcal{U}_0$ of $x$. Then there exists $n_0 \ge 1$ such that $X_n \cap \mathcal{U}_0 = \emptyset$ for all $n \ge n_0$, and thus $T^n(\mathcal{U}_0) \cap \mathcal{U}_0 = \emptyset$ for all $n \ge n_0$, since $T^n(\mathcal{U}_0) \subset X_n$.

Next, neither $x^+$ nor $x^-$ is periodic, since $X = \bigcap_{n=1}^{\infty} T^n(I)$. Because of this, for every $n \ge 1$, there exists a sufficiently small neighbourhood $\mathcal{U}_n$ of $x$ such that $T^i(\mathcal{U}_n) \cap \mathcal{U}_n = \emptyset$ for all $1 \le i \le n$. Let $\mathcal{U}_{n_0} \subset \mathcal{U}_0$ be such a sufficiently small neighbourhood. By our choice of $\mathcal{U}_0$, this means that $T^n(\mathcal{U}_{n_0}) \cap \mathcal{U}_{n_0} = \emptyset$ for all $n \ge 1$. Thus the point $x$ is not in the non-wandering set.

We now prove that $\overline{X} \setminus \{1\} \subseteq \Omega(T)$. Let $x$ be a point in $\overline{X} \setminus \{1\}$. By Lemma \ref{lem:x-structure}, $\overline{X} = A_1 \cup A_2$, where $A_1$ is a finite union of closed intervals and $A_2$ is a Cantor set (we may assume both are non-empty). Thus for any sufficiently small \textit{closed} neighbourhood $\mathcal{U}$ of $x$ exactly one of the following holds:
\begin{enumerate}
    \item $\mathcal{U} \cap \overline{X}$ contains exactly one closed interval $J$ such that $x \in J$ and $|J| > 0$;
    \item $\mathcal{U} \cap \overline{X}$ contains no interval.
\end{enumerate} 
In the first case, the first return map $R_J$ to $J$ is well-defined and bijective by Lemm\ \ref{lem:rj-facts}, so $x$ is clearly contained in $\Omega(T)$.

In the second case, we show that $x \in \omega(\beta)$ for some discontinuity $\beta \in \mathcal{C}$. This is sufficient, since all $\omega$-limit sets are contained in the non-wandering set, and thus $x \in \Omega(T)$ as well. Let $\mathcal{U}$ be a neighbourhood of $x$ as above. For any $n \ge 0$, the boundary points of $X_n$ consist of iterates of discontinuities. If $x \notin \omega(\beta)$ for every $\beta \in \mathcal{C}$, then the distance between the boundary of $\overline{X_n}$ and $x$ is larger than some positive constant $\delta$ for all sufficiently large $n$. Since $x$ is contained in every $\overline{X_n}$, an interval of length at least $\delta$ containing $x$ must also be contained in every $X_n$. Thus $\overline{X} \cap \mathcal{U}$ contains an interval of positive length, which is a contradiction with our choice of $\mathcal{U}$. Thus $x \in \omega(\beta)$ for some discontinuity $\beta \in \mathcal{C}$, which finishes the proof.
\end{proof}

\section{Stability, ACC and Matching}
\label{sec:stability-overview} 

\subsection{Definitions and examples} 
\label{subsec:stab=acc+m-dfn-exam}
We now recall the notation for the dynamics of a return map $R_J$ to $J$, from the beginning of Subsection \ref{subsec:basic-res-def}. The notation is slightly changed so that it better fits the notation of this section. Let $J = [x,y)$ be a dynamically non-trivial interval of $X$ and let $a_1, a_2, \dots a_{N-1}$ be the points in the interior of $J$ which land on discontinuities before returning to $J$ and let $l_1, l_2, \dots, l_N$ be the corresponding landing times. More precisely, let $l_i \ge 0$ be the smallest integer such that $T^{l_i}(a_i) = \beta_{J,i} \in \mathcal{C}$. Let $J_1 = [x,a_1^-), J_2 = [a_1^+, a_2^-), \dots, J_N = [a_{N-1}^+,y)$, and let $r_1, r_2, \dots, r_{N}$ be the return times to $J$ of these intervals, i.e.\ $R_J(J_i) = T^{r_i}(J_i)$. Note that $l_i < r_i$ and $l_i < r_{i+1}$. It will also be useful to define $a_0^+ \coloneqq x$ and $a_{N}^- = y$.

It is convenient to associate with each $J$ a permutation $\sigma_J$ which corresponds to the order on $J$ in which the subintervals return to $J$. for example, if $N=3$ and the order in which the intervals return is: $R_J(J_3) \; R_J(J_2) \; R_J(J_1)$, then $\sigma_J = (3 2 1)$. 

Let $\tau_J = (i_1, i_2, \dots, i_{N})$ be the inverse of $\sigma_J$, so that the order of the images under the return map is $R_J(J_{i_1}) \; R_J(J_{i_2}) \dots R_J(J_{i_{N}})$. Then the following equations are satisfied at the critical values of the return map $R_J$ (recall that $a \sim b$ denotes a pair of touching points from Definition \ref{def:touching}):

\[
\begin{array}{cc}
\label{eq:RJeq}
     &  R_J(a_{i_1}^-) \sim R_J(a_{i_2-1}^+) \\
     &  R_J(a_{i_2}^-) \sim R_J(a_{i_3-1}^+) \\
     & \dots \\
     & R_J(a_{i_{N-1}}^-) \sim R_J(a_{i_{N}-1}^+),
\end{array}
\]

These equations are central to the stability of the interval $J$, i.e.\ to the property that $J$ always has a continuation $\Tilde{J}$ for any sufficiently small perturbation. As before, by $\Tilde{Z}$ we will denote the continuation of any object of interest $Z$, e.g.\ $\Tilde{\beta}^+$, $\Tilde{X}$, $\Tilde{T}$ and similar.

Recall the definition of a stable $\ITM$ from the introduction:

\begin{definition}[Stable maps]
\label{def:stable} 
We say that a map $T \in \ITM(r)$ is \textit{stable} if there is a neighbourhood $\mathcal{U}$ of $T$ in $\ITM(r)$ such that:

\begin{enumerate}
\item The mapping assigning to each map $\Tilde{T} \in \mathcal{U}$ the corresponding set $\overline{X}(\Tilde{T})$ is continuous with respect to the Hausdorff topology on compact subsets of $[0,1]$. Moreover, for each $\tilde{T} \in \mathcal{U}$, $\overline{X}(\Tilde{T})$ is homeomorphic to $\overline{X}(T)$.
\item The number of discontinuities in $I \setminus \overline{X}(\Tilde{T})$ is constant in $\mathcal{U}$.
\end{enumerate}
\end{definition} 

The following is a simple fact:

\begin{lemma}
\label{lem:stable-fin-type}
Stable maps are of finite type.
\end{lemma}

\begin{proof}
Let $T$ be a stable map and let $\mathcal{U}$ be a neighbourhood of stability for $T$ from Definition \ref{def:stable}. All maps in $\mathcal{U}$ are either of finite or infinite type, since a Cantor set is not homeomorphic to a finite union of intervals. Thus it suffices to show that there is a finite type map in $\mathcal{U}$. Since $\ITM(r)$ is a polytope in $\mathbb{R}^{2r-1}$, the points with rational coordinates are clearly dense in it. Thus arbitrarily close to $T$, and therefore in $\mathcal{U}$, there is a map $T'$ with rational coefficients. We claim that $T'$ is of finite type. Let $\mathbb{N}(\gamma')$ be the set of all linear combinations with non-negative integer coefficients of the translation factors $\gamma'_1, \dots, \gamma'_r$ of $T'$. The minimal distance between two elements of this set is at least $\frac{1}{q^r}$, where $q$ is the largest denominator of the numbers $\gamma'_1, \dots, \gamma'_r$. Thus for any $x \in I$, the set $x + \mathbb{N}(\gamma') \cap [0,1)$ has cardinality at most $q^r$, which means that $x$ lands on a periodic point after at most $q^r$ iterates. Since there are only finitely many intervals of periodic points, $T^{q^r}(I) = T^{q^r+1}(I)$, so $T'$ is of finite type, which finishes the proof.
\end{proof}

This definition of stability is softer than the usual ones in one-dimensional dynamics, like $\Omega$-stability (see \cite{MR1239171}) and $J$-stability (see \cite{MR1312365}). The difference is that we do not require any dynamical conjugacy between the maps in the neighbourhood $\mathcal{U}$, just that two properties remain the same after perturbation: the number of discontinuities in and the shape of the non-wandering set. The reason is that it is impossible for all maps in any open set $\mathcal{U}$ of $\ITM(r)$ to be conjugate:

\begin{lemma}
\label{lem:itm-not-struct-stab}
No map in $\ITM(r)$ is structurally stable or $\Omega$-stable. 
\end{lemma} 
\begin{proof} Near each map $T$, there is a map with rational parameters, for which every point in $I$ lands on a periodic point by the proof of Lemma \ref{lem:stable-fin-type}, and one with rationally independent parameters, which has no periodic points. 
\end{proof} 

The next lemma follows immediately from the fact that the set $\Tilde{X}$ of a sufficiently small perturbation $\Tilde{T}$ of $T$ is close to and homeomorphic to $X$.

\begin{lemma}
\label{lem:conseq-stab}
For a sufficiently small perturbation $\Tilde{T}$ of a stable map $T$, $\overline{X}(\Tilde{T})$ consists of the same number of intervals as $X$ and these intervals are pairwise close to each other in the Hausdorff sense. This means that each component $J$ of $X$ has a well-defined continuation $\Tilde{J} \subset \Tilde{X}$, and in particular that its boundary points $x$ and $y$ also have continuations $\Tilde{x}$ and $\Tilde{y}$. Moreover, the continuation $\Tilde{J}$ of an interval $J$ contains the continuations of the discontinuities contained in $J$. \qed
\end{lemma}

Let us now give a non-trivial example of a stable map, as shown in Figure \ref{fig:stab-map} below.

\begin{example}[Stable map]
\label{ex:stable-map} 
Assume that $r=3$ and that the interval $J = [T^2(\beta_2^+),T(\beta_2^-)]$ 
contains $\beta_2$ in its interior. Let $J^-=[T^2(\beta_2^+),\beta_2^-]$ and $J^+=[\beta_2^+,T(\beta_2^-)]$.
Assume the following about the orbits of $\beta_2^+$ and $\beta_2^-$:

\begin{itemize}
    \item $T(\beta_2^+)\in I_1$ and $T^2(\beta_2^+) \in I_2$;
    \item $T(\beta_2^-)\in I_3$ and $T^2(\beta_2^-)\in I_1$;
\end{itemize} 
Then:

\begin{itemize}
    \item $T^3(\beta_2^+) = \beta_2^+  + \gamma_3 + \gamma_1 + \gamma_2$;
    \item $T^3(\beta_2^-) = \beta_2^-  + \gamma_2 + \gamma_3 + \gamma_1$,
\end{itemize}
so $T^3(\beta_2^+)\sim T^3(\beta_2^-)$. If we additionally assume that $T^3(\beta_2^+), T^3(\beta_2^-) \in J$, then for a suitable choice of parameters (as in Figure \ref{fig:stab-map}), $T$ is stable and $X(T)$ is equal to the orbit of $J$. Indeed, the itineraries of $\beta_2^+$ and $\beta_2^-$ remain the same for all nearby maps $\Tilde{T}$, which means that the interval $\Tilde{J} = [\Tilde{T}^2(\Tilde{\beta}_2^+),\Tilde{T}(\Tilde{\beta}_2^-))$ is a continuation of $J$. 
\end{example} 

\begin{figure}[h]
\begin{center}
\definecolor{qqzzff}{rgb}{0.,0.6,1.}
\definecolor{qqttzz}{rgb}{0.,0.2,0.6}
\definecolor{ffttww}{rgb}{1.,0.2,0.4}
\definecolor{uququq}{rgb}{0.25,0.25,0.25}
\definecolor{qqzzqq}{rgb}{0.,0.6,0.}
\begin{tikzpicture}[line join=round,x=8.0cm,y=8.0cm]
\clip(-0.1,-0.1) rectangle (1.1,1.1);
\draw [line width=1.0pt,color=uququq] (0.,0.)-- (1.,0.);
\draw [line width=1.0pt,color=uququq] (0.,1.)-- (0.,0.);
\draw [line width=1.0pt,dash pattern=on 2pt off 2pt,color=uququq] (0.,0.)-- (1.,1.);
\draw [line width=1.0pt,color=uququq] (1.,0.)-- (1.,1.);
\draw [line width=1.0pt,color=uququq] (1.,1.)-- (0.,1.);
\draw [line width=1.0pt,dash pattern=on 2pt off 2pt,color=ffttww] (0.4996686055531796,0.)-- (0.49966860555317955,0.6425257484103225);
\draw [line width=1.0pt,dash pattern=on 2pt off 2pt,color=ffttww] (0.49966860555317955,0.6425257484103225)-- (0.6425257484103225,0.6425257484103225);
\draw [line width=1.0pt,dash pattern=on 2pt off 2pt,color=ffttww] (0.66,0.)-- (0.66,0.8095238095238095);
\draw [line width=1.0pt,dash pattern=on 2pt off 2pt,color=ffttww] (0.66,0.8095238095238095)-- (0.8095238095238095,0.8095238095238095);
\draw [line width=1.0pt,dash pattern=on 2pt off 2pt,color=ffttww] (0.8095238095238095,0.8095238095238095)-- (0.8095238095238095,0.024594454816063052);
\draw [line width=1.0pt,dash pattern=on 2pt off 2pt,color=qqzzff] (0.166,0.)-- (0.166,0.166);
\draw [line width=1.0pt,dash pattern=on 2pt off 2pt,color=qqzzff] (0.166,0.166)-- (0.66,0.166);
\draw [line width=1.0pt,dash pattern=on 2pt off 2pt,color=qqzzff] (0.3102479928758915,0.)-- (0.3102479928758915,0.3102479928758915);
\draw [line width=1.0pt,dash pattern=on 2pt off 2pt,color=qqzzff] (0.3102479928758915,0.3102479928758915)-- (0.8102479928758916,0.3102479928758915);
\draw [line width=1.0pt,dash pattern=on 2pt off 2pt,color=qqzzff] (0.8102479928758916,0.3102479928758915)-- (0.8102479928758916,0.);
\draw [line width=1.0pt,dash pattern=on 2pt off 2pt,color=ffttww] (0.6425257484103225,0.6425257484103225)-- (0.6425257484103224,0.024594454816063052);
\draw [line width=1.0pt,dash pattern=on 2pt off 2pt,color=qqzzff] (0.166,0.)-- (0.166,0.5);
\draw [line width=1.0pt,dash pattern=on 2pt off 2pt,color=qqzzff] (0.3102479928758915,0.)-- (0.3102479928758915,0.6435813262092248);
\draw [line width=1.0pt,dash pattern=on 2pt off 2pt,color=qqzzff] (0.3102479928758915,0.6435813262092248)-- (0.6435813262092248,0.6435813262092248);
\draw [line width=1.0pt,dash pattern=on 2pt off 2pt,color=qqzzff] (0.6435813262092248,0.024594454816063052)-- (0.6435813262092248,0.6435813262092248);
\draw [line width=1.0pt,dash pattern=on 2pt off 2pt,color=qqzzff] (0.166,0.5)-- (0.5,0.5);
\draw [line width=1.0pt,dash pattern=on 2pt off 2pt,color=qqzzff] (0.5,0.5)-- (0.5,0.024594454816063052);
\draw [line cap=square,line width=3pt,color=qqttzz] (0.66,0.)-- (0.81,0.);
\draw [line cap=square,line width=3pt,color=qqttzz] (0.5,0.025)-- (0.64,0.025);
\draw [line cap=square,line width=3pt,color=qqttzz] (0.166,0.)-- (0.31,0.);
\draw [line cap=square,line width=3pt,color=ffttww] (0.64,0.025)-- (0.81,0.025);
\draw [line cap=square,line width=3pt,color=ffttww] (0.5,0.)-- (0.66,0.);
{\scriptsize
\draw (0.59,0.0) node[anchor=north] {$J^-$};
\draw (0.75,0.0) node[anchor=north] {$J^+$};
\draw (0.25,0.0) node[anchor=north] {$T(J^+)$};
\draw (0.57,0.09) node[anchor=north] {$T^2(J^+)$};
\draw (0.74,0.09) node[anchor=north] {$T(J^-)$};
}
\draw [line width=2pt,color=qqzzqq] (0.66,0.166)-- (1.,0.5);
\draw [line width=2pt,color=qqzzqq] (0.,0.33)-- (0.33,0.66);
\draw [line width=2pt,color=qqzzqq] (0.33,0.476)-- (0.66,0.81);
\end{tikzpicture}
\caption{{A particular map as in Example ~\ref{ex:stable-map} with $\beta_0=0$, $\beta_1=1/3$, $\beta_2=2/3$, $\beta_3=1$, $\gamma_1=1/3$, $\gamma_2=1/7$, $\gamma_3=-1/2$. The intervals $J^-$ and $T(J^-)$ are shown in pink and the intervals $J^+$, $T(J^+)$, $T^2(J^+)$ are drawn in blue.}}
\label{fig:stab-map}
\end{center}
\end{figure}
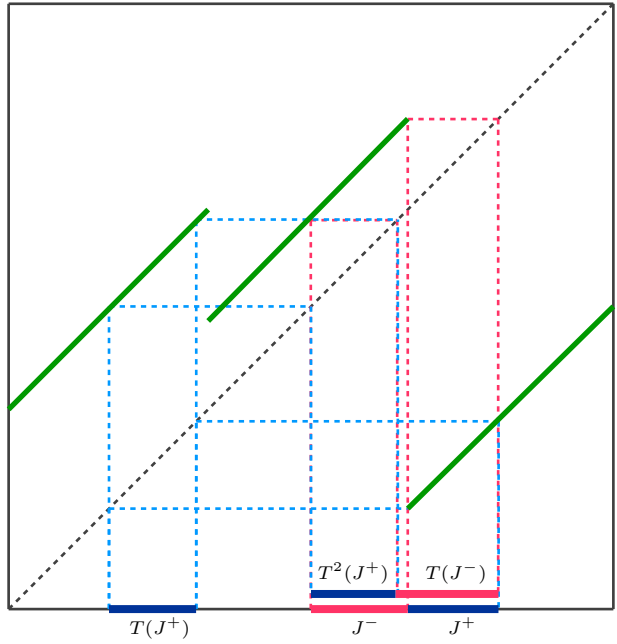

The definition of stability mainly concerns the topology and the location of the non-wandering set, giving only limited information about the dynamics on $X$. Moreover, the dynamical information retained by making a small perturbation of a stable map is relatively soft. We would therefore like to characterise stability in terms of more concrete dynamical properties. Intuitively speaking, a map $T$ is not stable if there exist arbitrarily small perturbations $\Tilde{T}$ that cause at least one of the following three things to happen:

\begin{enumerate}
    \item A definite part of $X$ gets removed, i.e.\ the assignment of $\overline{X}$ is not lower semi-continuous at $T$. More formally, $\Tilde{X}$ has a smaller number of intervals than $X$, or some interval component of $\tilde{X}$ is shorter than the corresponding interval component of $X$ by some positive constant. 
    \item A definite new part gets added to $X$, i.e.\ the assignment of $\overline{X}$ is not upper semi-continuous at $T$. More formally, $\Tilde{X}$ has a larger number of intervals than $X$, or some interval component of $\tilde{X}$ is longer than the corresponding interval component of $X$ by some positive constant. 
    \item Discontinuities that were in $X$ are not in $\Tilde{X}$. 
\end{enumerate}
We are therefore looking for dynamical properties that guarantee that none of these three things can happen for an arbitrarily small perturbation of $T$.
To guarantee that $(1)$ does not happen, we can ask for the dynamics of the full orbit of every interval $J$ of $X$ to remain the same after perturbation (as in the example above). For $(3)$, we can simply require that there are no discontinuities in the boundary of $X$ (assuming the first and the second things do not happen). Item $(2)$ is the most elusive, and we will need to discuss it a bit before giving the definitions of the required dynamical properties.

The following example illustrates the simplest way in which the set $X$ can become larger. The dynamics of a particular map as in Example \ref{ex:ghost-preimage} is shown in Figure \ref{fig:ghost-preimage}.

\begin{example}
\label{ex:ghost-preimage}
Assume that there are two discontinuities $\betas$ and $\betass$ for which there exist $k_1 > 0$ and $k_2 > 0$ such that:
\[
T^{k_1}(\betas^-) = \betass^- \text{ and } T^{k_2}(\betass^+) = \betas^+.
\]
Moreover, assume that none of the discontinuities $\betas^-, \betas^+, \betass^-$ and $\betass^+$ are in $X$ and that there exists a perturbation $\Tilde{T}$ of $T$ such that:
\begin{itemize}
    \item The itinerary of $\betas^-$ remains the same up to time $k_1$ and $T^{k_1}(\betas^-) = \betass^- + \epsilon$;
    \item The itinerary of $\betass^+$ remains the same up to time $k_2$ and $T^{k_2}(\betass^+) = \betas^+ - \epsilon$.
\end{itemize}
There exists an interval of points to the left of $\betas^-$ all of which have the same itinerary as $\betas^-$ up to time $k_1$. The same holds for $\betass^+$ and its itinerary up to time $k_2$. In particular, these intervals map forward continuously up to times $k_1$ and $k_2$, respectively. Thus for a sufficiently small $\epsilon$ the interval $[\betas^- - \epsilon,\betas^-)$ maps forward continuously up to time $k_1 + k_2$. It is therefore periodic with period $k_1 + k_2$ and contained in $\Tilde{X}$. As $\betas^-, \betas^+, \betass^-$ and $\betass^+$ were by assumption not in $X$, their distance from $X$ is bounded from below by a positive constant. Therefore $\overline{X}(\Tilde{T})$ and $\overline{X}(T)$ cannot be close for a sufficiently small perturbation. Thus $T$ is not a stable map.
\end{example}

\begin{figure}
    \centering
    \includegraphics[width=\linewidth]{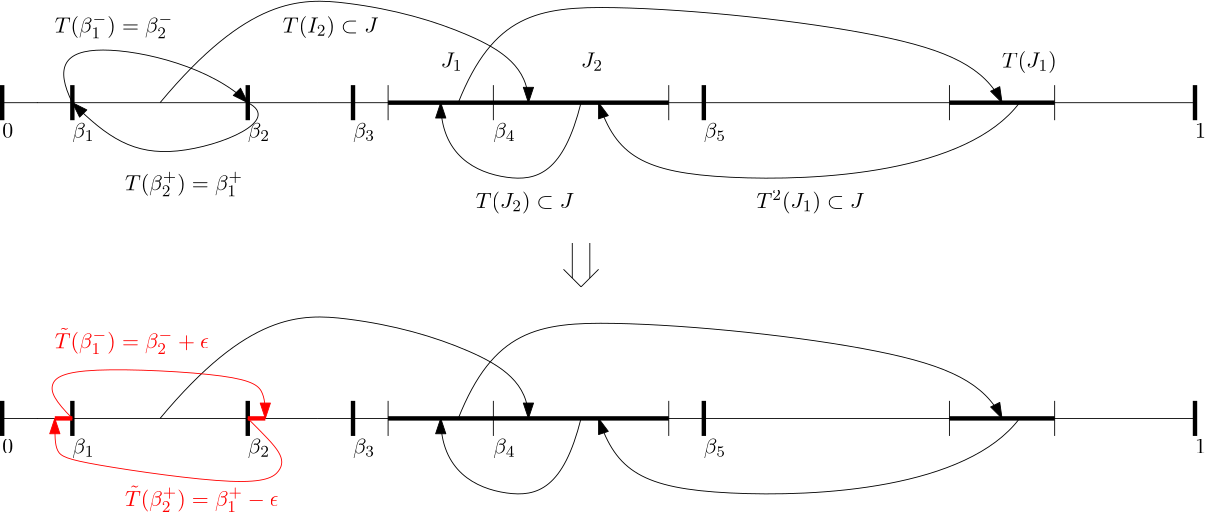}
    \caption{A map $T$ as in Example \ref{ex:ghost-preimage}, with $k_1 = k_2 = 1$, $\beta_* = \beta_1$ and $\beta_{**} = \beta_2$. The perturbation increases $\gamma_1$ by $\epsilon$, decreases $\gamma_3$ by $\epsilon$ and leaves the other parameters unchanged.}
    \label{fig:ghost-preimage}
\end{figure}

This example is the simplest case of a more general phenomenon called \textit{ghost preimages}. We give the definition for a $+$-type discontinuity, with the one for $-$-type case being analogous.

\begin{definition}[Ghost preimage]
Let $\betas^+$ be a discontinuity of $T$. A discontinuity $\betass^-$ that lands on $\betas^-$ is called a \textit{ghost preimage} of $\betas^+$.
\end{definition}

Thus a ghost preimage of a $+$-type discontinuity $\betas^+$ is a $-$-type discontinuity $\betass^-$ that lands on $\betas^-$. The definition can easily be extended to a $-$-type discontinuity. The motivation for the name is the following: $\betass^-$ is \textit{almost} a preimage of $\betas^+$, i.e.\ by an arbitrarily small perturbation we can make the iterate $T^k(\betass^-)$ land to the right of $\betas^+$, thus creating an actual preimage of $\betas^+$.


In Example \ref{ex:ghost-preimage} the point $\betas^+$ is a ghost preimage of $\betass^-$ and that $\betass^-$ is a ghost preimage of $\betas^+$ as well. This is the property that allows for the enlargement of $X$, and it motivates the following definition:

\begin{definition}[Ghost tree]
\label{def:ghost-tree}
Let $\beta$ be a discontinuity of $T$. The \textit{ghost tree} $\mathcal{GT}(\beta)$ of $\beta$ is defined inductively in the following way. Set $\beta$ to be the root of the tree, i.e.\ the set of level $0$ vertices of the tree. Assume that we have defined all of the vertices of level $\le n$ and all of the edges between them. Then the level $n+1$ vertices of the tree correspond to the set of all ghost preimages (if such exist) of the discontinuities corresponding to level $n$ vertices of the tree. The new edges are those between discontinuities and their ghost preimages, i.e.\ a new direct edge $\betas \leftarrow \betass$ is added if and only if $\betas$ is a level $n$ vertex, $\betass$ is a level $n+1$ vertex and $\betass$ is a ghost preimage of $\betas$. 
\end{definition}

Thus $\mathcal{GT}(\beta)$ is a directed tree, with well-defined levels, and edges exist only between vertices of consecutive levels. If $\betas \leftarrow \betass$, then we say that $\betas$ is a \textit{parent} of $\betass$ and that $\betass$ is a \textit{child} of $\betas$. The types of discontinuities alternate between consecutive levels and the tree can have finitely or infinitely many levels. The following simple lemma characterises infinite ghost trees:

\begin{lemma}[Infinite ghost trees]
\label{lem:inf-gt}
The ghost tree $\mathcal{GT}(\beta)$ of some discontinuity $\beta$ is infinite if $\beta$ appears more than once as a vertex in $\mathcal{GT}(\beta)$.
\end{lemma}

This is exactly the case in Example \ref{ex:ghost-preimage}.

\begin{proof}
If $\beta$ appears again as a vertex on some level $n > 0$, then all of the levels that appeared before the second appearance of $\beta$ have to repeat. There is at least one level between these appearances, as the discontinuity type has to change from level to level. Thus the tree is infinite.
\end{proof}

We are now ready to state the two dynamical properties that characterise stability. All of our definitions apply to all interval components $J$ of $X$, except in the special case when no point in the interior of $J$ lands on a discontinuity and the return map $R_J$ to this component is the identity. We will call such components of $X$ \textit{dynamically trivial}, and components not of this form \textit{dynamically non-trivial}. Any interval that is referred to as dynamically trivial or a non-trivial is assumed to be an interval component of $X$. 

\begin{definition}[Absence of Critical Connections (ACC)]
\label{def:acc}
We say that a finite type $T$ satisfies the \textit{ACC} condition if the following three conditions hold: 
\begin{enumerate}
    \item (A1) For every interval component $J$ of $X$ and each point $a \in J$, the orbit of $a$ up to and not including the return time to $J$ contains at most one critical point of $T$;
    \item (A2) For every dynamically non-trivial interval $J$, none of the boundary points of $J$ land on discontinuities up to and not including the return time to $J$;
    \item (A3) For every discontinuity $\beta \notin X$, its ghost tree $\mathcal{GT}(\beta)$ does not contain $\beta$.
\end{enumerate}
\end{definition}

An example of an interval $J$ satisfying A1 and A2 is shown in Figure \ref{fig:acc}.

\begin{figure}
    \centering
    \includegraphics[width=\linewidth]{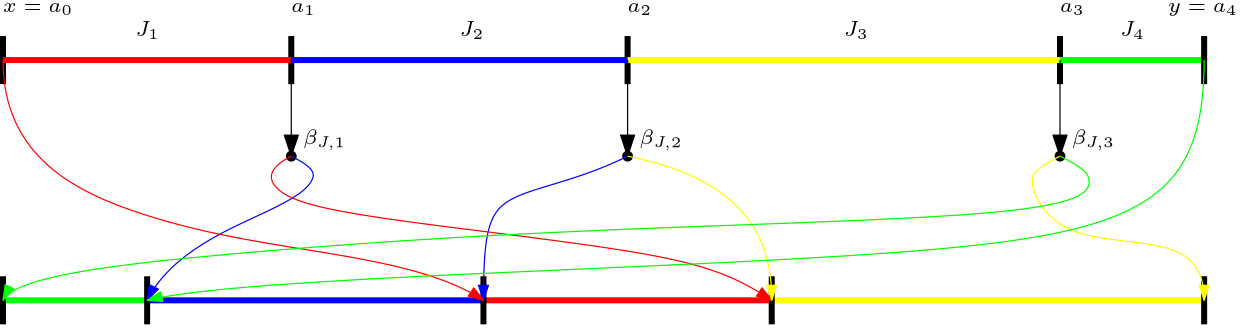}
    \caption{Full orbit of an interval $J$ for which the return map $R_J$ has four continuity intervals and satisfies A1 and A2.}
    \label{fig:acc}
\end{figure}

The reason for the name is that if any of these properties is violated, then the defining parameters $(\gamma \, \beta)$ of $T$ satisfy a non-trivial linear equation which we call a \textit{critical connection}. Note that these are the particular critical connections that we want to avoid, and that we do not avoid all of them. Note that A2 implies that there are no discontinuities in the boundary of $X$, as landing times by definition include zero. 

If A1 and A2 hold, then every vertex of the ghost tree $\mathcal{GT}(\beta)$ for a $\beta \notin X$ corresponds to a discontinuity that is not contained in $X$. Indeed, by the above, there are no discontinuities in the boundary. Thus if $\beta_2^-$ is a ghost preimage of $\beta_1^+$, and $\beta_2^-$ is contained in $X$, this means that $\beta_1^+$ is also contained in $X$. Thus every ghost preimage of a $\beta \notin X$ is also not contained $X$, and the claim follows by induction. 

\begin{definition}[Matching]
\label{def:matching}
We say that a finite type $T$ satisfies the \textit{Matching} condition if for every dynamically non-trivial interval $J$, exactly one point $a$ in the interior of $J$ lands on a critical point before returning to $J$.
\end{definition}
The reason we call this property Matching is that if $a$ is the single point from above, then we have that:

\begin{enumerate}
    \item $J = [R_J(a^+), R_J(a^-))$;
    \item $R^2_j(a^-) \sim R^2_J(a^+)$,
\end{enumerate}
and thus the second iterates under $R_J$ of $a^+$ and $a^-$ `match'. Figure \ref{fig:matching} illustrates the dynamics of an interval $J$ that satisfies Matching. In this case, it is allowed for the point $a$ to land on multiple discontinuities before returning to $J$ and for the boundary points $x$ and $y$ to land on discontinuities. Recall from the beginning of subsection \ref{subsec:vec+pert-lemma} that $m_{j}^{\pm}$ is the number of discontinuities that $a_{j}^{\pm}$ lands on before returning to $J$.

\begin{figure}[h]
    \centering
    \includegraphics[width=0.9\linewidth]{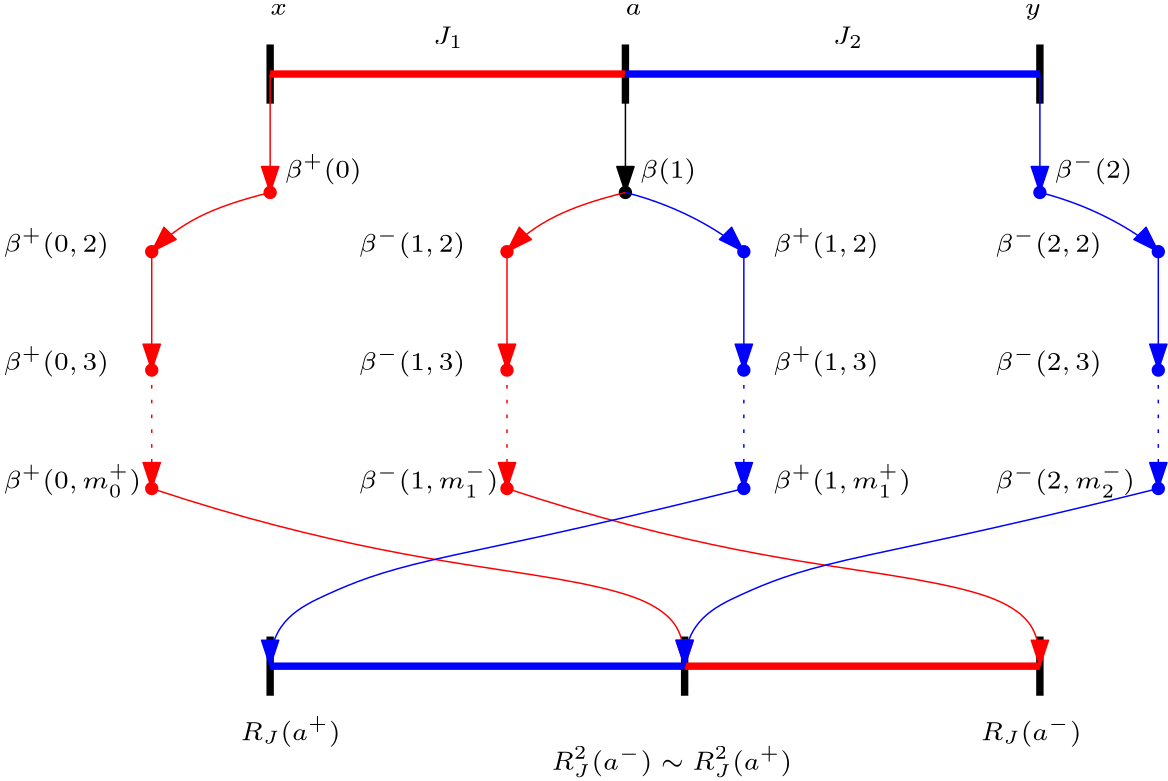}
    \caption{Full orbit of an interval $J$ satisfying Matching.}
    \label{fig:matching}
\end{figure}

An analogous property, also called the Matching, has been discussed in slightly different contexts (see \cite{MR3893724}, \cite{MR3597033}, \cite{MR2422375}). If $T$ satisfies Matching, then the return map to every interval component $J$ of $X$ is either a rotation or the identity, which is a very restrictive condition.

With all of the definitions out of the way, we can now recall the statement of our main theorem:

\begin{theorem}[Characterisation of Stability]
\label{thm:stability=acc+m}
A finite type interval translation map $T$ is stable if and only if it satisfies the ACC and Matching properties.
\end{theorem}

The proof of Theorem \ref{thm:stability=acc+m} is contained in the next Section \ref{sec:proof-stability=acc+m}. In subsection \ref{subsec:acc+m=stab-pf-overview} we give an informal overview of the proof, while the proof is spread over subsections \ref{subsec:s+a1+a2->m}, \ref{subsec:s->acc} and \ref{subsec:acc+m->stability}.

\section{Characterisation of Stability}
\label{sec:proof-stability=acc+m}

\subsection{Overview of the proof of Theorem \ref{thm:stability=acc+m}}
\label{subsec:acc+m=stab-pf-overview}

It turns out to be surprisingly tricky to show the `only if' direction of Theorem \ref{thm:stability=acc+m}, i.e.\ that a map violating either ACC or Matching is not stable. To prove that a map is not stable, one needs to produce arbitrarily small perturbations that discontinuously change $X$. A natural way to do this is to perturb the map at the places where it violates either ACC or Matching.

Consider for example the return map to an interval component $J$ as in in Figure \ref{fig:a1 violation}, for which the point $a_2^+$ violates A1 and lands on two critical points $\beta^+$ and $\beta_*^+$.

\begin{figure}[h]
    \centering
    \includegraphics[width=0.85\linewidth]{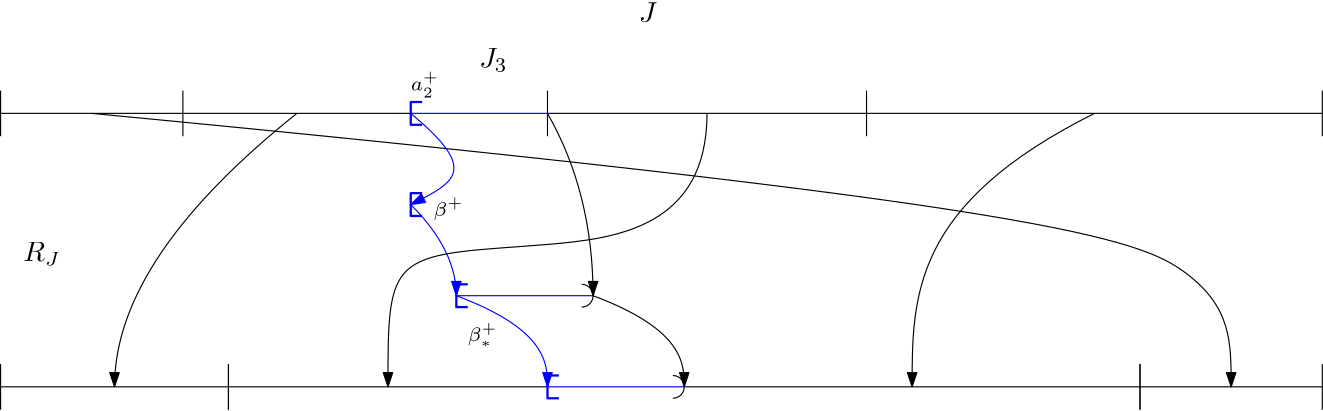}
    \caption{An example of a map for which an interval component $J$ of $X$ violates A1 at the point $a_2^+$. The dynamics of the point $a_2^+$ and the corresponding interval are in blue.}
    \label{fig:a1 violation}
\end{figure}

One would expect that perturbing the parameters $(\gamma \, \beta)$ in such a way that $\beta^+$ does not land on $\beta_{*}^+$ would result in a hole for the return map to $R_J$, as in Figure \ref{fig:pert-a1}, which would change the topology of $X$: the interval $J$ should split into (at least) two smaller intervals. In the several figures that follow, we will mark the dynamics that change under perturbation in \textcolor{red}{red}, and the remaining dynamics by other colours.

\begin{figure}[h]
    \centering
    \includegraphics[width=0.85\linewidth]{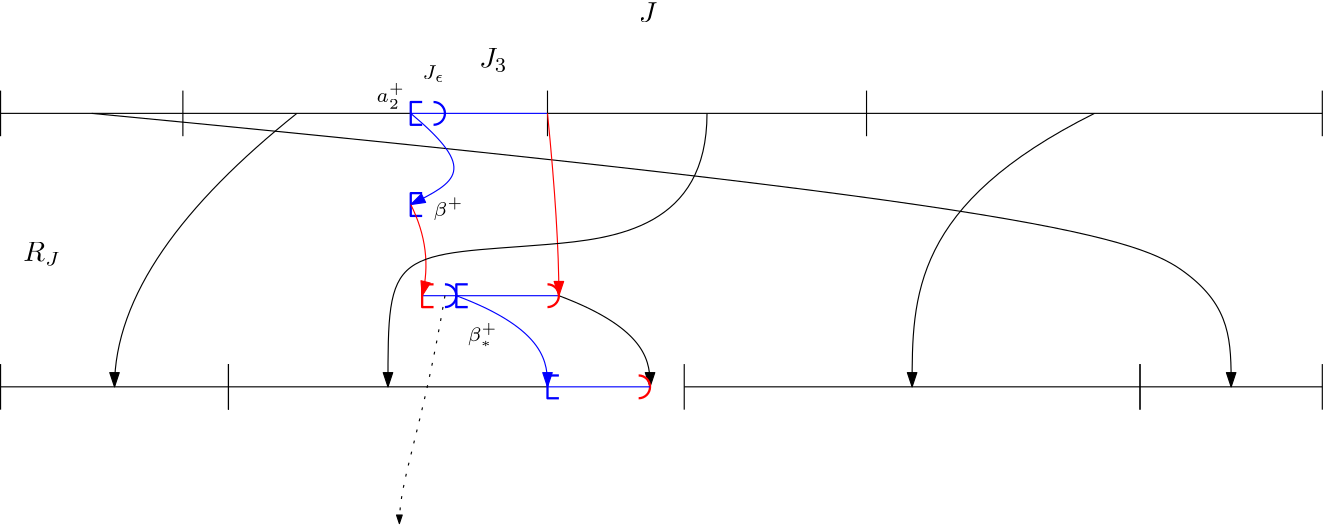}
    \caption{The red arrows indicate the dynamical changes after perturbation: the orbit of $a_2^+$ has moved to the left of $\beta_*^+$. Assuming that the remaining dynamics of the return map stay the same, it is expected that now only the smaller interval $J_3 \setminus J_\epsilon$ returns to $J$, creating a hole.}
    \label{fig:pert-a1}
\end{figure}

However, an arbitrarily small change in the parameters $(\gamma \, \beta)$ possibly affects the long-term behaviour of $\beta^-_*$ as well. This is possible if $\beta^-_*$ lands on some other discontinuity $\beta_{**}^-$. Our perturbation might have then caused $\beta^-_*$ to now land to the right of $\beta_{**}$, even if the perturbation is arbitrarily small. If $\beta_{**}^+$ also landed into $J$ near the right endpoint of $R_J(J_3)$, it is possible that $J_{\epsilon}$ now lands back into $J$ to fill the hole created by the change in the orbit of $J_3$. All of this is illustrated in Figure \ref{fig:pert-a1-j-preserved}.

\begin{figure}
    \centering
    \includegraphics[width=0.85\linewidth]{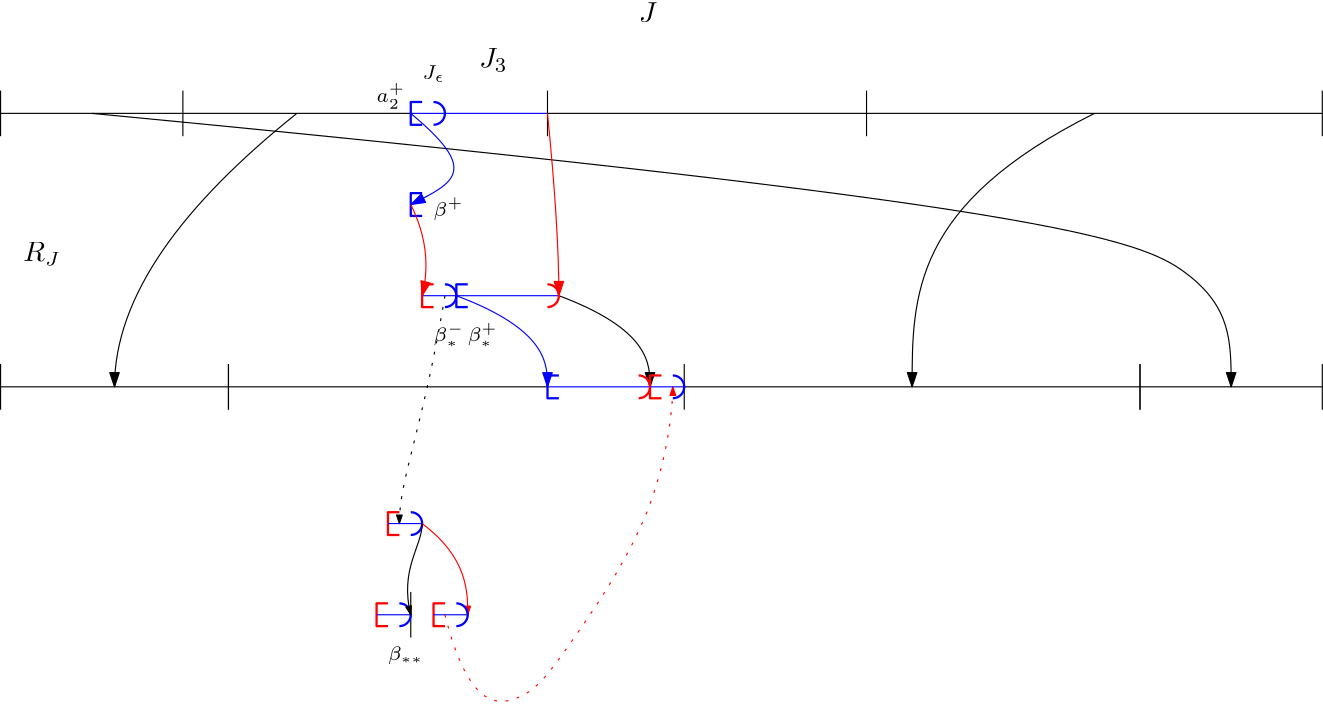}
    \caption{The perturbation could also cause the orbit of $J_{\epsilon}$ to eventually return to $J$ and fill the hole from Figure \ref{fig:pert-a1}.}
    \label{fig:pert-a1-j-preserved}
\end{figure}

This example is a special case of a general phenomenon that perturbing critical connections for the purpose of changing the topology of $X$ may cause itinerary changes at other critical connections, which then still lead to a continuous change in $X$. In general, there is no good way to rule out that this happens for every natural choice of explicit perturbation coming from the violation of either ACC or Matching. On the other hand, there is also no good way to track the dynamics of the return map $R_J$ to $J$ if we choose some generic non-explicit perturbation, which means we are unable to show that stability gets violated in this case either. 

In view of the discussion above, instead of showing that a stable map satisfies ACC and Matching, we adopt an indirect approach and first assume that the map is stable \textit{and satisfies conditions A1 and A2}. We are then able to show that the only way such a map retains the topological structure of $X$ under perturbation is if the dynamics of return maps to interval components of $X$ remain the same (Proposition \ref{prop:stab-return-map}). In particular, this rules out changes in the return map as in Figure \ref{fig:pert-a1-j-preserved}. This then forces the map to also satisfy Matching (Proposition \ref{prop:exactly-1}), but even this turns out to be non-trivial, as we now explain.

Recall that $k_s(x,n)$ is the number of entries of the orbit of $x$ to the interval $I_s$ up to, but not including, time $n$:
\[
k_s(x,n) \coloneqq \# \{ T^j(x) \in I_s \text{ for } 0 \le j < n \}.
\]
We have the following explicit formula for any iterate of $x$:
\[
T^n(x) = x + \sum_{s=1}^r k_s(x,n) \gamma_s.
\]

Consider an interval component $J$ of $X$ for which the return map has $N$ continuity intervals. By Proposition \ref{prop:stab-return-map}, the dynamics of the return map to $J$, i.e.\ the itineraries of continuity intervals up to the return time to $J$, remain the same. Consider the equation corresponding to the $j$'th critical value of $R_J$ (recall the notation introduced at the beginning of Subsection \ref{subsec:stab=acc+m-dfn-exam}):

\[
\begin{array}{cc}
     &R_J(a_{i_j}^-) \sim R_J(a_{i_j-1}^+)  \\[1mm]
     \iff &T^{r_{i_j}}(a_{i_j}^-) \sim  T^{r_{i_j-1}}(a_{i_j-1}^+) \\[1mm]
     \iff &T^{r_{i_j}-l_{i_j}}(\beta_{i_j}^-) \sim  T^{(r_{i_j-1})-(l_{i_j-1})}(\beta_{i_j-1}^+) \\[1mm]
     \iff &\beta_{i_j} + \sum_{s=1}^r k_s(\beta_{i_j}^-, r_{i_j} - l_{i_j}) \gamma_s = \beta_{i_j-1} + \sum_{s=1}^r k_s(\beta^+_{i_j-1}, r_{i_j-1}-l_{i_j-1}) \gamma_s.
\end{array}
\]

If the dynamical structure of $R_J$ remains the same for every sufficiently small perturbation, then all these equations also remain true for any sufficiently small perturbation. In particular, they remain the same if we change only a single parameter $\beta_s$, for $1 \le s \le r-1$, or $\gamma_s$, for $1 \le s \le r$. This forces the following equalities:
\[
\begin{array}{cc}
     \beta_{i_j} = \beta_{i_j-1} \text{ for all } 1 \le j < N; \\
     k_s(\beta_{i_j}^-, r_{i_j} - l_{i_j}) = k_s(\beta_{i_j-1}^+,r_{i_j-1}-l_{i_j-1}) \text{ for all } 1 \le j < N \text{ and } 1 \le s \le r.
\end{array} 
\]

Although this is a very restrictive condition, there is a priori no reason why it should not be possible. It is here where one needs to invoke the Linear Dependence of Return Map Vectors Lemma (Lemma \ref{lem:lin-dep-rj}). More precisely, by putting all of these equalities together in an appropriate way, one can show that the dynamical vectors associated to $J$ violate the form of linear dependence forced by Lemma \ref{lem:lin-dep-rj}, unless $N \le 2$, which shows Matching. We have now shown that Stability, A1 and A2 imply Matching (Theorem \ref{thm:s+A1+A2=M}), so what remains to show is that Stability implies ACC (Theorem \ref{thm:s-acc}).

To show that a stable map satisfies A3, we use the Perturbation Lemma \ref{lem:pert-lem} to produce an explicit perturbation that enlarges the set $X$ in a discontinuous way when A3 is violated. As A1 and A2 are generic conditions, arbitrarily near any stable map, there are stable maps satisfying A1 and A2, and therefore Matching. Matching implies that the orbits of different discontinuities never pass through the same interval components of $X$. For sufficiently small perturbations, this separation property for orbits of discontinuities can be `pulled back' to the original map. This implies properties A1 and A2, which completes the proof.

For the `if' direction (Theorem \ref{thm:acc+m->stability}), consider a map $T$ that satisfies A1, A2 and Matching. Then by Matching, the return map to every dynamically non-trivial interval component $J = [x,y)$ of $X$ is isomorphic to a circle rotation. Let $a$ be the single point in the interior of $J$ that lands on a discontinuity of $T$. Then by A1 and A2, $a$ is the only point in $J$ that lands on a discontinuity, and $a$ lands only on one discontinuity of $T$ before returning to $J$ (see Figure \ref{fig:pert-a1-a2-m}). Let $\beta$ denote this discontinuity of $T$, and let $l$ be the landing time of $a$ to $\beta$, so that $a := T^{-l}(\beta)$. Let $J_1 = [x,a)$ and $J_2 =[a,y)$ be the two continuity intervals of $J$, and let $r_1$ and $r_2$ be the return times of $J_1$ and $J_2$ to $J$, respectively. Thus $x^+ = T^{
r_2}(a^+)$ and $y^- = T^{r_1}(a^-)$. We can now consider the equation at the only critical value of $R_J$:

\[
\begin{array}{cc}
     &R_J^2(a^-) \sim R_J^2(a^+)  \\[1mm]
     \iff &T^{r_2 + r_1}(a^-) \sim  T^{r_1 + r_2}(a^+) \\[1mm]
     \iff &a + \sum_{s=1}^r k_s(a^+, r_2 + r_1) \gamma_s = a + \sum_{s=1}^r k_s(a^-,r_1 + r_2) \gamma_s \\[1mm]
     \iff &0 = 0,
\end{array}
\]

where the last equality follows from the fact that the itinerary of $a^-$ up to time $r_2 + r_1$ is equal to the itinerary of $J_1$ up to time $r_1$ followed by the itinerary of $J_2$ up to time $r_2$, while the itinerary of $a^+$ up to time $r_1 + r_2$ is equal to the itinerary of $J_2$ up to time $r_2$ followed by the itinerary of $J_1$ up to time $r_1$.

Because $T$ satisfies A1 and A2, for a sufficiently small perturbation, there will still be a point in the interior of $\Tilde{J}$ that lands on $\Tilde{\beta}$ after $l$ iterates, and both the $-$ and $+$ part of this point will have the same itinerary up to time $r_1 + r_2$. Thus the structure of the return map will remain the same, so the interval $J$ moves continuously under perturbation (see Figure \ref{fig:pert-a1-a2-m}).

Thus the set $X$ moves at least lower-semicontinuously under perturbation. The difficult part is then to show that the only way $X$ can discontinuously increase is if there exists an infinite ghost tree, which would violate A3. This is done by considering the set of itinerary changes for critical points not contained in $X$ that are possible for arbitrarily small perturbations. It can be shown that such itinerary changes can produce new interval components of $X$ for arbitrarily small perturbations if and only if some discontinuity has an infinite ghost tree, which completes the proof. 

\begin{figure}[th]
    \centering
    \includegraphics[width=0.7\linewidth]{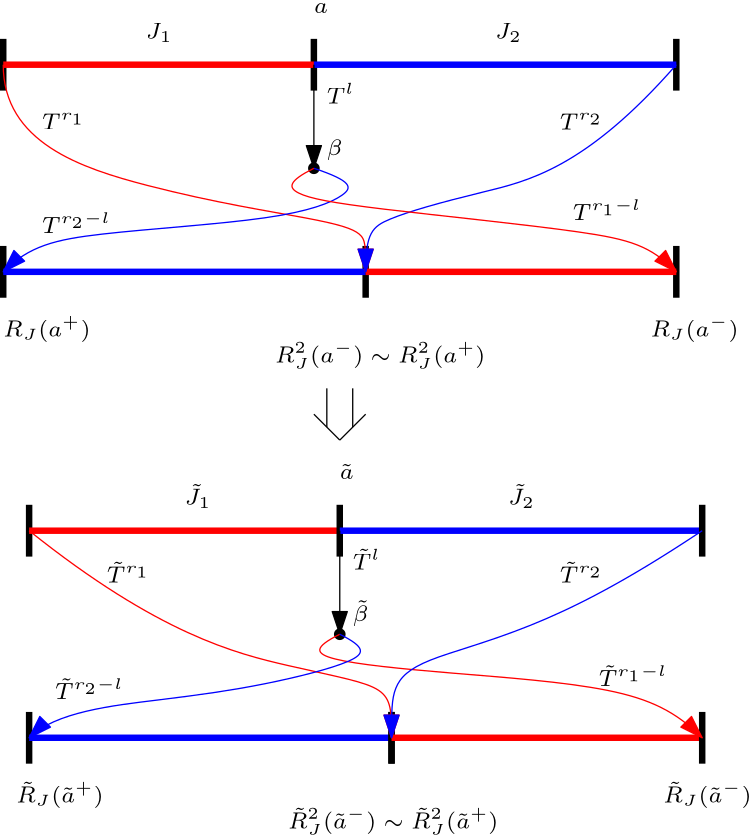}
    \caption{Perturbation of a return map $R_J$ satisfying A1, A2 and Matching. The endpoints of the interval $J$ and the point $a$ move continuously, and the dynamical structure of $R_J$ remains the same.}
    \label{fig:pert-a1-a2-m}
\end{figure}

\subsection{Stability, A1 and A2 imply Matching}
\label{subsec:s+a1+a2->m}

The main property that A1 and A2 allow for is control over finite time itineraries for any sufficiently small perturbation. To make this precise, consider any point $x \in I$ that does not land on a critical point before time $n$. Let $s(x,i)$, for $0 < i < n$, be index such that  the iterate $T^i(x)$ is contained in $I_{s(x,i)}$. Associated to each iterate $T^i(x)$ we therefore have the following two vectors:

\begin{align*}
v^{i,+} &:= \left(\sum_{1 \le s \le r} i_s(x) \, \bm{e}_s, \,  -\bm{f}_{s(x,i)-1}\right); \\
v^{i,-} &:= \left(\sum_{1 \le s \le r} i_s(x) \, \bm{e}_s, \, -\bm{f}_{s(x,i)}\right).
\end{align*}

Thus we have the following formula for the distance between an iterate of $x$ and the critical points in $I_{s(x,i)}$:

\begin{align*}
T^i(x) - \beta_{s(x,i)-1}^+ = x + \langle v^{i,+}, (\gamma \, \beta) \rangle; \\
\beta_{s(x,i)}^- - T^i(x) = -\langle v^{i,+} , (\gamma \, \beta) \rangle - x.
\end{align*}

By the assumption on time $n$, all of the quantities above are positive for $0 < i < n$. Thus they remain positive for all sufficiently small perturbation $\Tilde{T}$ of $T$, because they clearly depend continuously on $(\gamma \, \beta)$. This means that the itinerary of $x$ up to time $n$ remains the same for all sufficiently small perturbations. If the perturbations we make make are arbitrarily small, then the itineraries up to time $n$ do not change for a whole interval of points containing $x$. This follows immediately from the fact that if some orbit up to time $n$ of some point $x$ does not land on a discontinuity, then the same holds for an interval $J$ containing $x$ in its interior. Thus if we take a slightly smaller interval $J'$ containing $x$, then its itinerary up to time $n$ will not change for all sufficiently small perturbations.

We would like to have the same conclusions for itineraries of critical points. For this we have to slightly change the argument because the points themselves change under perturbation. Assume that some discontinuity $\beta$ does not land on a discontinuity before time $n$. Then analogously as above we define $s(\beta,i)$, for $0 < i < n$, to be index such that  the iterate $T^i(\beta)$ is contained in $I_{s(\beta,i)}$. Again, we have the following vectors:

\begin{align*}
v^{i,+}(\beta) &:= \left(\sum_{1 \le s \le r} i_s(\beta) \, \bm{e}_s, \, \bm{f}_{\text{ind}(\beta)} -\bm{f}_{s(x,i)-1}\right); \\
v^{i,-}(\beta) &:= \left(\sum_{1 \le s \le r} i_s(\beta) \, \bm{e}_s, \, \bm{f}_{\text{ind}(\beta)} -\bm{f}_{s(x,i)}\right),
\end{align*}

and the following formula for the distance from the critical points in $I_{s(x,i)}$:

\begin{align*}
T^i(\beta) - \beta_{s(x,i)-1}^+ = \langle v^{i,+}, (\gamma \, \beta) \rangle; \\
\beta_{s(x,i)}^- - T^i(\beta) = -\langle v^{i,+} , (\gamma \, \beta) \rangle .
\end{align*}

Thus by the same argument as before, the itinerary of $\beta$ up to time $n$ does not change for all sufficiently small perturbations.

We can analogously do this for the backward itinerary, i.e.\ a choice of consecutive $T$-preimages, of any point $x$ or discontinuity $\beta$, so it also persists under sufficiently small perturbations. If we assume A2, then the itinerary of every boundary point $x$ of $X$ also remains the same after a sufficiently small perturbation. Similarly as before, there is an interval $J'$ containing $x$ in its interior for which the itinerary does not change for all sufficiently small perturbations.

We will now put together these observations to prove that, under the assumptions of stability, A1 and A2, the structure of the entire return map $R_J$ to any dynamically non-trivial interval $J$ remains stable for a sufficiently small perturbation. More precisely, we will prove that the number of points that land on discontinuities remains the same, that the images of $J_1, \dots, J_N$ under $R_J$ remain in the same order and that their return times remain the same. We first give a lemma that eliminates some annoying cases in the proof of Proposition \ref{prop:stab-return-map}.

\begin{lemma}
\label{lem:not-fixed}
Let $T$ be a stable map that also satisfies property A2. Then $R_J(J_1) \neq J_1$ and $R_J(J_{N}) \neq J_{N}$, for any dynamically non-trivial interval $J$. In other words, the return map cannot fix the boundary subintervals. 
\end{lemma}

\begin{proof}
We give the proof for $J_1$, with the one for $J_{N}$ being analogous. Assume the opposite, that $R_J(J_1) = J_1$. Then $R_J(x) = x$, i.e.\ $x$ is periodic. By Lemma \ref{lem:J-dynamics}, $x = T^m(\beta^+)$ for some $\beta^+ \in X$. As $T$ is of finite type, $T\vert_X$ is a bijection, and thus $\beta^+$ must be in the forward orbit of $x$. Thus $x$ must land on $\beta^+$ before it returns to $J$, which contradicts A2. 
\end{proof}

\begin{proposition}[Stability of the return map]
\label{prop:stab-return-map}
Let $T$ be a stable map that also satisfies properties A1 and A2, and let $\Tilde{T}$ be a sufficiently small perturbation of $T$. Let $J = [x, y)$, where $x = T^{m_1}(\beta_{J,\tau(1)-1}^+)$ and $y = T^{m_2}(\beta_{J,\tau(N)}^-)$, be a dynamically non-trivial interval and let $\Tilde{J}$ be its continuation. Then the following properties hold:
\begin{enumerate}
    \item $\Tilde{J} = [\Tilde{T}^{m_1}(\Tilde{\beta}_{J,\tau(1)-1}^+),\Tilde{T}^{m_2}(\Tilde{\beta}_{J,\tau(N)}^-))$;
    \item The only points in $\Tilde{J}$ that land on critical points before they return are $\Tilde{a}_1, \dots, \Tilde{a}_{N-1}$, where $\Tilde{a}_i := \Tilde{T}^{-l_i}(\Tilde{\beta}_{J,i})$ for the appropriate backwards branch of $\Tilde{T}$ and $1 \le i \le N-1$;
    \item Each $\Tilde{a}_i^{+}$ has the same itinerary up to time $r_{i+1}$ as before the perturbation, for $0 \le i \le N-1$. Analogously for and $\Tilde{a}_i^{-}$ and $r_i$, for $1 \le i \le N$;
    \item The return time of each interval $[\Tilde{a}_i^+, \Tilde{a}_{i+1}^-)$ is $r_{i+1}$ (where we use the convention that $\Tilde{a}_0^+ = \Tilde{x}$ and $\Tilde{a}_{N}^- = \Tilde{y}$);
    \item The order in $\Tilde{J}$ in which the intervals return is given by the same permutation $\tau = (j_1, j_2, \dots, j_{N})$ as for $J$.
\end{enumerate}
\end{proposition}

How small the perturbation $\Tilde{T}$ needs to be will be clear from the proof.

\begin{proof}
Let $\Tilde{T}$ be sufficiently small, as discussed at the beginning of subsection \ref{subsec:s+a1+a2->m}, so that the backward $T^{-1}_{\vert X}$ itinerary of $\beta_{J,i}$ until time $l_i$ remains the same for every $1 \le i \le N$. We may also assume $\Tilde{T}$ is sufficiently small so that the itinerary of every $\beta_{J,i}^{+}$ up to time $r_{i+1}$, and $\beta_{J,i}^{-}$ up to time $r_{i}$ remains the same after perturbation. Thus the points $\Tilde{a}_i := \Tilde{T}^{-l_i}$ are well-defined, contained in $\Tilde{J}$ and the itinerary of every $\Tilde{a}_i^{+}$, resp. $\Tilde{a}_i^{-}$, for $1 \le i \le N$ remains unchanged until time $r_{i+1}$, resp. $r_i$, when it lands back into $\Tilde{J}$. 
Analogously, by the A2 assumption, we may assume that the perturbation $\Tilde{T}$ is sufficiently small so that the points $\Tilde{a}_0^+ := \Tilde{T}^{-m_1}(\Tilde{\beta}^+_{J,0})$ and $\Tilde{a}_N^- := \Tilde{T}^{-m_2}(\Tilde{\beta}^-_{J,N})$ are well defined, and that their itineraries remain the same up to times $r_1$ and $r_N$, respectively. Moreover, by A2, the same holds for small intervals of constant length containing these points. In particular, it holds for the boundary points $\Tilde{x}^+$ and $\Tilde{y}^-$ of $\Tilde{J}$. Thus we have shown that the interval $[\Tilde{T}^{m_1}(\Tilde{\beta}_{J,\tau(1)-1}^+),\Tilde{T}^{m_2}(\Tilde{\beta}_{J,\tau(N)}^-))$ has a well-defined first return map, so it is contained in $\Tilde{J}$. By Lemma \ref{lem:rj-facts}, this return map is a bijection. Thus if we had $\Tilde{x}^+ \neq \Tilde{T}^{m_1}(\Tilde{\beta}_{J,\tau(1)-1}^+)$, the interval between these two points is not contained in $[\Tilde{T}^{m_1}(\Tilde{\beta}_{J,\tau(1)-1}^+),\Tilde{T}^{m_2}(\Tilde{\beta}_{J,\tau(N)}^-))$, but maps into the interior of this interval by Lemma \ref{lem:not-fixed}, and therefore it would not be contained in $\overline{X}$. Analogously for $\Tilde{y}^-$ and $\Tilde{T}^{-m_2}(\Tilde{\beta}^-_{J,N})$, and therefore we have that $\Tilde{J} = [\Tilde{T}^{m_1}(\Tilde{\beta}_{J,\tau(1)-1}^+),\Tilde{T}^{m_2}(\Tilde{\beta}_{J,\tau(N)}^-))$, so property $1.$ follows.

Because the entire interval of points $[\Tilde{a}_i^+, \Tilde{a}_{i+1}^-)$ has the same itinerary up to time $r_{i+1}$, for all $0 \le i < N$, there are no other points in $\Tilde{J}$ that land on discontinuities before returning to $\Tilde{J}$, so property $2.$ follows. Properties $3.$, $4.$ and $5.$ now easily follow from the previous paragraph.
\end{proof}

Proposition \ref{prop:stab-return-map} clearly implies that the equations for the return map still hold after perturbation:

\begin{corollary}
\label{cor:stab-eq}
Let $T$ and $J$ be as in Proposition \ref{prop:stab-return-map}. Then for any sufficiently small perturbation of $T$, all of the following equations still hold:
\[
\begin{array}{cc}
\label{eqn:pert-RJeq}
     &  \Tilde{R}_J(\Tilde{a}_{i_1}^-) \sim \Tilde{R}_J(\Tilde{a}_{i_2-1}^+) \\
     &  \Tilde{R}_J(\Tilde{a}_{i_2}^-) \sim \Tilde{R}_J(\Tilde{a}_{i_3-1}^+) \\
     & \dots \\
     & \Tilde{R}_J(\Tilde{a}_{i_{N-1}}^-) \sim \Tilde{R}_J(\Tilde{a}_{i_{N}-1}^+).
\end{array}
\] \qed
\end{corollary}

With Proposition \ref{prop:stab-return-map} and Corollary \ref{cor:stab-eq}, we are ready to prove that any dynamically non-trivial interval $J$ of $X$ contains exactly one discontinuity.

\begin{proposition}
\label{prop:exactly-1}
Let $T$ be a stable map that satisfies the A1 and A2 conditions, and let $J$ be a dynamically non-trivial interval of $X$. Then the return map $R_J$ to $J$ has exactly one discontinuity.
\end{proposition}

\begin{proof}
Assume first that the return map has $3$ or more discontinuities. Then one of the equations above:
\[
\Tilde{R}_J(\Tilde{a}_{i_j}^-) \sim \Tilde{R}_J(\Tilde{a}_{i_{j+1}-1}^+)
\]
is such that neither $\Tilde{a}_{i_j}^-$ nor $\Tilde{a}_{i_{j+1}-1}^+$ are boundary points of $\Tilde{J}$. Thus we have that:
\[
\Tilde{a}_{i_j}^- + \sum_{s=1}^r k_s(a_{i_j}^-, r_{i_j}) \Tilde{\gamma}_s = \Tilde{a}_{i_{j+1}-1}^+ + \sum_{s=1}^r k_s(a_{i_{j+1}-1}^+, r_{i_{j+1}}) \Tilde{\gamma}_s
\]
for any sufficiently small perturbation $\delta$. Let us first assume that the discontinuities $\Tilde{a}_{i_j}^-$ and $\Tilde{a}_{i_{j+1}-1}^+$ are different, i.e.\ that $i_j \neq i_{j+1}-1$.

Now, let the perturbation $\Tilde{T}$ be such that we only change $\beta_{J,i_j}$ to $\Tilde{\beta}_{J,i_j} \coloneq \beta_{i_j} + \epsilon$, where $\epsilon$ is sufficiently small so that Proposition \ref{prop:stab-return-map} applies. By Corollary \ref{cor:stab-eq}, we have that:
\begin{align*}
\Tilde{a}_{i_j}^- + \sum_{s=1}^r k_s(a_{i_j}^-, r_{i_j}) \Tilde{\gamma}_s &= \Tilde{a}_{i_{j+1}-1}^+ + \sum_s k_s(a_{i_{j+1}-1}^+, r_{i_{j+1}}) \Tilde{\gamma}_s \\
\implies \epsilon &= 0,
\end{align*}
as the only value that changes is $a_{i_j}^-$, since $\Tilde{a}_{i_j}^- := \Tilde{T}^{-l_{i_j}}(\Tilde{\beta}_{J,i_j}^-)$ by definition. This is clearly a contradiction, so we may assume $i_j = i_{j+1}-1$. Then if we make a perturbation that only changes some $\gamma_t$ to $\Tilde{\gamma}_t := \gamma_t + \epsilon$, we get:
\begin{align*}
&\epsilon\left(-k_t(a_{i_j}^-, l_{i_j}) + k_t(\beta_{J,i_j}^-, r_{i_j} - l_{i_j})\right) = \\
&\epsilon\left(-k_t(a_{i_{j}}^+, l_{i_{j}}) + k_t(\beta_{J,i_{j}}^+, r_{i_{j}+1} - l_{i_{j}})\right),
\end{align*}
by Corollary \ref{cor:stab-eq}. Since the above holds for all $\epsilon$ and $t$, we get:
\[
k_t(\beta_{J,i_j}^-, r_{i_j} - l_{i_j}) = k_t(\beta_{J,i_{j}}^+, r_{i_{j}+1} - l_{i_{j}}),
\]
since $k_t(a_{i_j}^-, l_{i_j}) = k_t(a_{i_{j}}^+, l_{i_{j}})$, as these points have the same itinerary up to the landing time $l_{i_j}$ to $\beta_{J,i_{j}}$. Since this holds for all $t$, the return vectors $v_{i_j}$ and $v_{i_j+1}$, defined as:
\begin{align*}
v_{i_j} &:= \left(\sum_{1 \le s \le r} k_s(a_{i_j-1}^+, r_{i_j}) \, \bm{e}_s, \,  0\right) \\
v_{i_j+1} &:= \left(\sum_{1 \le s \le r} k_s(a_{i_j}^+, r_{i_j+1}) \, \bm{e}_s, \,  0\right)
\end{align*}
of $J_{i_j}$ and $J_{i_j+1}$ are equal. We claim that this is a contradiction with the Linear Dependence of Return Map Vectors Lemma \ref{lem:lin-dep-rj}. Indeed, by construction:
\begin{align*}
v_{i_j} &= L_{i_j-1} + \sum_{k-1}^{m_{i_j-1}-1}C^{+}(i_j-1,k) + R_{i_j-1}^{+} \\
v_{i_j+1} &= L_{i_j} + \sum_{k-1}^{m_{i_j}}C^{+}(i_j,k) + R_{i_j}^{+}.
\end{align*}
Thus by setting $\alpha_{i_j-1} = \alpha^{+}(i_j-1,k) = \alpha_{i_j-1}^{+} = 1$ and $\alpha_{i_j} = \alpha^{+}(i_j,k) = \alpha_{i_j}^{+} = -1$, and all of the other coefficients to zero, we get a linear dependence between return map vectors that is different than the one forced by Lemma \ref{lem:lin-dep-rj}, which is a contradiction. Thus $R_J$ does not have more than two discontinuities.

Assume now that $R_J$ has exactly two discontinuities. By Lemma \ref{lem:not-fixed}, $\sigma_J = (3 2 1)$. The above equations now reduce to:
\begin{align*}
&a_1^+ + \sum_{s=1}^r k_s(a_1^+, r_2)\gamma_s = y + \sum_{s=1}^r k_s(y, r_3)\gamma_s \\
&a_2^- + \sum_{s=1}^r k_s(a_2^-, r_2)\gamma_s = x + \sum_{s=1}^r k_s(x, r_1)\gamma_s,
\end{align*}
where $x = T^{r_3-l_2}(\beta_{J,2}^+)$ and $y = T^{r_1-l_1}(\beta_{J,1}^-)$. Thus by doing the same perturbations as above, we get that:
\begin{align*}
-k_t(a_1^+, l_1) + k_t(\beta_{J,1}^+, r_2-l_1) &= k_t(\beta_{J,1}^-, r_3 + r_1 - l_1) \\
-k_t(a_2^-, l_2) + k_t(\beta_{J,2}^-, r_2-l_2) &= k_t(\beta_{J,2}^+, r_1 + r_3 - l_2),
\end{align*}
for all $1 \le t \le r$. If we define the vectors $v_1, v_2$ and $v_3$ as above, this implies $v_1 + v_3 = v_2$. Then by setting $\alpha_{0}^{+} = 1$, $\alpha_{1} = \alpha^{+}(1,k) = \alpha_{1}^{+} = -2$ and $\alpha_{2} = \alpha^{+}(2,k) = \alpha_{2}^{+} = 1$ we again get a different linear dependence than the one forced by Lemma \ref{lem:lin-dep-rj}, which is a contradiction.

Thus, it is impossible for $J$ to contain more than one discontinuity. As it is dynamically non-trivial, it must contain exactly one discontinuity.
\end{proof}

Proposition \ref{prop:exactly-1} tells us that every dynamically non-trivial interval $J$ contains exactly one point that lands on a discontinuity before returning $J$, which implies Matching. Thus we have shown the following:

\begin{theorem}[Stability + A1 + A2 $\implies$ Matching]
\label{thm:s+A1+A2=M}
A stable map that satisfies properties A1 and A2 also satisfies the Matching property. \qed
\end{theorem}

\subsection{Stability implies ACC}
\label{subsec:s->acc}

Because of Theorem \ref{thm:s+A1+A2=M}, we now only need to prove that stability implies ACC to complete the first implication of Theorem \ref{thm:stability=acc+m}. Our proof will be indirect: we will first perturb the stable map $T$ to a stable map $\tilde T$ that satisfies ACC, and therefore Matching, and use this property of $\tilde T$ to conclude that $T$ also satisfies ACC. 

We now recall the definition of rational independence:

\begin{definition}[Rational independence]
Let $\{ s_1, s_2, \dots, s_n \}$ be a finite set of real numbers. These numbers are said to be rationally independent if the following holds:
\[
\left( \sum_i q_i s_i = 0, \; \text{with} \; q_i \in \mathbb{Q} \right) \implies \left( q_i = 0 \; \forall i \right).
\]
\end{definition}

This condition on the set $\{ \beta_1, \dots, \beta_{r-1}, \gamma_1, \dots, \gamma_r \}$ is stronger than ACC. Indeed, if ACC does not hold, i.e.\ if one of A1, A2 or A3 is violated, then we have an equation of the following form:
\[
\beta + \sum_{s=1}^r n_s \gamma_s = \betas,
\]
which implies that the the numbers $\{ \beta_1, \dots, \beta_{r-1}, \gamma_1, \dots, \gamma_r \}$ are rationally dependent.

Rational independence, and therefore ACC, can be achieved by an arbitrarily small perturbation. This follows from the next lemma applied to the set  $\{ \beta_1, \dots, \beta_{r-1}, \gamma_1, \dots, \gamma_r \}$, whose proof is straightforward by induction.

\begin{lemma}
\label{rat-ind}
Let $\{ s_1, s_2, \dots, s_n \}$ be a finite set of real numbers. Then we can make an arbitrarily small perturbation to each of the numbers so that they become rationally independent. \qed
\end{lemma}

We first establish two lemmas about maps that satisfy properties A1, A2 and Matching.

\begin{lemma}
\label{lem:indep-orb}
Let $T$ be a finite type map that satisfies properties A1, A2 and Matching. Then the orbit of any point $z \in X$ enters at most one interval component of $X$ that contains a discontinuity of $T$.
\end{lemma}

\begin{proof}
Assume the contrary, that for some point $z$ we have that $T^{k_1}(z) \in J_1 = [x_1,y_1) \ni \beta$ and $T^{k_2}(z) \in J_2 = [x_2,y_2) \ni \betas$, where $J_1$ and $J_2$ are two disjoint interval components of $X$, and $k_1 < k_2$. Without loss of generality, we may assume that $T^{k_1}(z)$ lands into $J_2$ before returning to $J_1$. Indeed, we can choose $k_1$ to be the last time $z$ is in $J_1$ before landing to $J_2$. Moreover, we may assume that $T^{k_2}(z)$ lands to $J_1$ before returning to $J_2$, by choosing $k_2$ as the last time $T^{k_1}(z)$ is in $J_2$ before returning to $J_1$. We may assume $T^{k_1}(z) \in [\beta^+, y_1)$, with the other case being analogous. The the image of the interval $[\beta^+, y_1)$ at time $k_2 - k_1$ is contained in $J_2$, by our choice of $k_1$ and $k_2$. By Matching, no point except $\beta^+$ in the interior of this interval lands on a discontinuity before returning to $J_1$, $\beta^+$ does not land because of A1 and $y$ does not land because of A2. Thus the image $T^{k_2-k_1}([\beta^+, y_1))$ is compactly contained in the interior of one of the intervals $[x_2, \betas^-), [\betas^+, y_2)$. For simplicity, we denote that interval by $J'$. As $J_2$ is dynamically non-trivial, by Matching and A1 no point in $J'$ can land on discontinuity before returning to $J_2$. By our choice of $k_2$, the interval $T^{k_2-k_1}([\beta^+, y_1))$ lands to $J_1$ before it returns to $J_2$. At this time $\beta^+$ lands on $x$, but the points from $J'$ to the left of $T^{k_2-k_1}(\beta^+)$ land outside of $J$. Since they are contained in $X$, this contradicts the maximality of $J$.
\end{proof}

\begin{lemma}
\label{lem:triv-int-struct}
Assume that a finite type map $T$ satisfies properties A1, A2 and Matching. Then every dynamically trivial interval $J = [x,y)$ of $X$ has the property that $x$ lands only on the discontinuity $\beta^+$ and $y$ lands only on the discontinuity $\beta^-$, for some $\beta \in X$.
\end{lemma}

\begin{proof}
Assume first that $x$ lands on $\betas^+$ at the same time $y$ lands on $\betass^-$. This means that $\betas^+$ and $\betass^-$ are contained in the same interval of $X$. Since by A1 and A2 there are no discontinuities in the boundary of $X$, the return map to this interval of $X$ has at least two discontinuities of its return map. This is a contradiction with Matching, so this case is impossible.

Let $T^{k_1}(x) = \betas^+$ and $T^{k_2}(y) = \betass^-$, and assume without loss of generality that $k_1 < k_2$. By A1 and A2, the interval $J_1$ of $X$ containing $T^{k_1}(J)$ is dynamically non-trivial, which by Matching implies that no point in $J_1$, except $\betas$, lands on a discontinuity before returning to $J_1$. But $\betas^+$ and $\betas^-$ also do not land on any other discontinuities because of A1. As $y$ lands on a discontinuity, this forces $\betass^- = \betas^-$.
\end{proof}

We are now ready to prove the following theorem, which alongside Theorem \ref{thm:s+A1+A2=M} proves the first implication of Theorem \ref{thm:stability=acc+m}:

\begin{theorem}
\label{thm:s-acc}
A stable map $T$ satisfies the ACC condition.
\end{theorem}

\begin{proof}
Let us first derive some simpler properties of $T$. Let $\Tilde{T}$ be a sufficiently small perturbation of $T$ which satisfies properties A1 and A2. Then $\Tilde{T}$ also satisfies the Matching condition, by Theorem \ref{thm:s+A1+A2=M}, and therefore also the assumptions of Lemma \ref{lem:indep-orb} and Lemma \ref{lem:triv-int-struct}. Moreover, every interval component $\Tilde{J}$ of $\Tilde{X}$ contains at most one discontinuity in its interior and no discontinuities in its boundary. Since the number of discontinuities in $\overline{X}(\Tilde{T})$ is constant in the neighbourhood of stability $\mathcal{U}$ of $T$, the continuity of $\overline{X}(\Tilde{T})$ implies the same conclusion for $X$: it also has no discontinuities in the boundary, and every interval component $J$ of $X$ contains at most one discontinuity of $T$, and $J$ contains a discontinuity $\beta$ if and only if $\Tilde{J}$ contains $\Tilde{\beta}$.

Assume now that $T$ does not satisfy A1: some dynamically non-trivial interval $J = [x,y)$ contains a point $a$ which lands on a discontinuity more than once before returning to $J$. By the paragraph above, we know that $a$ must be in the interior of $J$ and that no other point in $J$ lands on a discontinuity before returning to $J$. We may without loss of generality assume that $a^+$ lands on two discontinuities, with the case for $a^-$ being analogous. Thus there are times $k_1 < k_2$ such that $T^{k_1}(a^+) = \betas^+ \in J_1$ and $T^{k_2}(a^+) = \betass^+ \in J_2$, where $J_1 \neq J_2$ are interval components of $X$. Let $z \in [a^+,y)$ not land on a critical point before returning to $J$ and assume that the perturbation is small enough so that $z$ has the same itinerary until it returns to $\Tilde{J}$. By Lemma \ref{lem:indep-orb}, the $\Tilde{T}$-orbit of $z$ lands into at most one interval component of $\Tilde{X}$ that contains a discontinuity, before it returns to $\Tilde{J}$. The same holds for the $T$-orbit of $z$ before returning to $J$, as the $T$ and $\Tilde{T}$ orbits of $z$ pass through corresponding intervals of $X$ and $\Tilde{X}$, respectively. Moreover, this also holds for every point in the continuity interval $[a,y)$ of $R_J$ containing $z$. Thus $a^+$ does not land in both $J_1$ and $J_2$, which is a contradiction. Thus $T$ satisfies A1.  

Next, assume that $T$ violates A2: there is a point $x$ in the boundary of a dynamically non-trivial interval $J = [x,y)$ of $X$ that lands on a discontinuity before returning to $J$ (the case for $y$ is analogous). Let $z$ be a point in the interior of $J$, to the left of $a$, the only discontinuity of $R_J$. In particular, the orbit of $z$ up to return time does not land on any critical point. We may thus assume that the perturbation is small enough so that $z$ has the same itinerary up to return time with respect to the intervals of $X$ and Lemma \ref{lem:indep-orb} holds for $z$ and $\Tilde{T}$. Thus the $T$-orbit of $z$ has the same properties: it enters at most one interval of $X$ that contains a discontinuity of $T$. Therefore this has to hold until return time for the continuity interval $[x,a)$ of $R_J$ that contains $z$. Thus it is impossible that both $x$ and $a$ land on a discontinuity, which is a contradiction. Thus $T$ satisfies A2 as well.

Finally, assume that $T$ violates A3. We will produce arbitrarily small perturbations of $T$ such that the set $\Tilde{X}$ contains more discontinuities than $X$, contradicting stability. If A3 does not hold, then there is some $\beta$ that appears more than once as a vertex in its ghost tree $\mathcal{GT}(\beta)$. Without loss of generality, we may assume that it is of $+$-type and we label it $\beta_{i_1}^+$. Thus we have a finite path in the ghost tree of the following form: $\beta_{i_1}^+ \leftarrow \beta_{i_2}^- \leftarrow \beta_{i_3}^+ \leftarrow \dots \leftarrow \beta_{i_{2k+1}}^+ = \beta_{i_1}^+$. Since there are no discontinuities in the boundary of $X$, no discontinuity in this path is contained in $X$. For even $j$ with $2 \le j \le 2k$, let $k_j$ be the time at which $\beta_{i_j}^+$ lands on $\beta_{{i_{j-1}}}^+$. Define $k_j$ analogously for odd $j$ with $3 \le j \le 2k+1$.

Using part (b) of the Perturbation Lemma \ref{lem:pert-lem}, there exists an arbitrarily small perturbation $\Tilde{T}$ of $T$ such that the following holds. For every landing $\beta_{i_j}^+ \to \beta_{{i_{j-1}}}^+$, for even $j$ with $2 \le j \le 2k$, we set the difference $\Tilde{T}^{k_j} ( \Tilde{\beta}_{i_j}^+) - \Tilde{\beta}_{{i_{j-1}}}^+ = -\epsilon$, where $\epsilon > 0$ is arbitrarily small. Note that if $\beta_{i_j}^+$ lands on other discontinuities before landing to $\beta_{{i_{j-1}}}^+$, we may simply preserve all of these landings, also by part (b) of the Perturbation Lemma \ref{lem:pert-lem}. We may analogously set all the difference $\Tilde{T}^{k_j} ( \Tilde{\beta}_{i_j}^+) - \Tilde{\beta}_{{i_{j-1}}}^+ = \epsilon$ for odd $j$ with $3 \le j \le 2k+1$. By the Perturbation Lemma \ref{lem:pert-lem}, we may preserve all other landings outside of $X$, while the dynamics on $X$ remain unchanged since $T$ is stable and satisfies A1, A2 and Matching. By construction, $\Tilde{\beta}_{i_1}^+$ is periodic with period $\sum_{j=1}^{2k+1} k_j$, and is therefore contained in $\Tilde{X}$, which gives the desired contradiction.
\end{proof}

\subsection{ACC + Matching implies Stability}
\label{subsec:acc+m->stability}
We are now going to prove the other direction in Theorem \ref{thm:stability=acc+m}.

\begin{theorem}
\label{thm:acc+m->stability}
Let $T \in \ITM(r)$ be a finite type map that satisfies the ACC and Matching conditions. Then $T$ is a stable map.
\end{theorem}

In the proof, we will encounter the following situation. For all sufficiently small perturbations, the itinerary of a point $z$ does not change up to some finite time $n$ when it lands into the interior of a set $S$. Then we may also assume the perturbations are sufficiently small so that there is an interval of positive length containing $z$ whose itinerary up to time $n$ does not change, and which lands into the interior of $S$ at time $n$. If the point $z$ is a discontinuity, then this interval will contain $z$ in its left or right boundary, depending on whether $z$ is of $+$-type or of $-$-type. 

\begin{proof}

We will show that for a sufficiently small perturbation of $T$, $\Tilde{X}$ is close to and homeomorphic to $X$, and contains the same discontinuities as $X$. By definition, this shows that $T$ is stable.

By A1 and A2, there are no discontinuities in the boundary of $X$. Moreover, for any $T$ that satisfies A1, A2 and Matching, every dynamically non-trivial interval $J$ has the property that only a single point $a$ in the interior of $J$ lands on a discontinuity and that the return map to $J$ is a rotation. This gives $X = \bigsqcup_{\beta \in X \cap \, \mathcal{C}} O(J_\beta)$, where $J_{\beta}$ is the dynamically non-trivial interval of $X$ containing $\beta$. The orbits $O(J_{\beta})$ are disjoint and there is a positive distance between them because of Lemma \ref{lem:indep-orb}.

By assumption, every dynamically non-trivial interval $J$ is of the form $[T^{k_2}(a^+), T^{k_1}(a^-))$, where $R_J = T^{k_1}$ on $J_1 = [x,a^-)$ and $R_J = T^{k_2}$ on $J_2 = [a^+,y)$. Moreover, for every such interval $J$ of $X$, there is an interval $J' = [x', y')$ containing $J$ such that:
\begin{itemize}
    \item No point in $[x',\beta^-)$ lands on discontinuity up to time $k_1$ and no point in $[\beta^+, y')$ lands on a discontinuity up to time $k_2$.
    \item $J$ is compactly contained in the interior of $J'$, i.e.\ $x' < x$ and $y' > y$.
\end{itemize}
Indeed, this follows from A2. We may assume that $J'$ is sufficiently small so that for a sufficiently small perturbation of $T$, the itineraries of points in $[x',\Tilde{a}^-)$ up to time $k_1$ and $[\Tilde{a}^+, y')$ remain the same as for $T$. Moreover, we may assume that the perturbation is small enough so that the itineraries of $a^+$ and $a^-$ up to time $k_1 + k_2$ remain the same as for $T$. Thus the interval $[\Tilde{T}^{k_2}(\Tilde{a}^+), \Tilde{T}^{k_1}(\Tilde{a}^-))$ has a well-defined return map on it that is a rotation, so it is in particular contained in $\Tilde{X}$. We will show that it must be an interval of $\Tilde{X}$. For a sufficiently small perturbation, it is compactly contained in the interior of $J'$. This means that the points in $[x',\Tilde{T}^{k_2}(\Tilde{a}^+))$ land into $[\Tilde{T}^{k_2}(\Tilde{a}^+), \Tilde{T}^{k_1}(\Tilde{a}^-))$ at time $k_1$ and points in $[\Tilde{T}^{k_1}(\Tilde{a}^-),y')$ land into $[\Tilde{T}^{k_2}(\Tilde{a}^+), \Tilde{T}^{k_1}(\Tilde{a}^-))$ after time $k_2$, which means that they are not contained in $\Tilde{X}$. Thus $[\Tilde{T}^{k_2}(\Tilde{a}^+), \Tilde{T}^{k_1}(\Tilde{a}^-))$ is an interval of $\Tilde{X}$, so it is in fact equal to $\Tilde{J}$.

For every dynamically trivial interval $J$ of $X$, there is a single dynamically non-trivial interval $J_{\beta}$ of $X$ that contains a discontinuity $\beta$ of $T$ into which $J$ lands, by Lemma \ref{lem:indep-orb} and Lemma \ref{lem:triv-int-struct}. Thus it is contained in the orbit of $J_{\beta}$. Let $J_1$ and $J_2$ be the two continuity intervals of $J_{\beta}$. As $J$ is dynamically trivial, it is equal to an iterate of either $J_1$ or $J_2$. We may assume the former, with the other case being analogous. Assume $J = T^{k_1}(J_1)$. By the discussion for dynamically non-trivial intervals, this iterate moves continuously for sufficiently small perturbations. Thus $\Tilde{T}^{k_1}(\Tilde{J}_1)$ is an interval contained in $\Tilde{X}$. We now show that it is maximal. Indeed, there is again an interval $J' \supset J$ such that $J$ is compactly contained in the interior of $J'$ and the itineraries of $J'$ and $J$ are equal up to the time $J$ lands into $J_{\beta}$. We may assume $J'$ is sufficiently small so that it lands into the interior of $J_{\beta}'$ and that this also holds for all sufficiently small perturbations. As above, this means that the points in $J' \setminus \Tilde{T}^{k_1}(\Tilde{J}_1)$ are not contained in $\Tilde{X}$, which shows that $\Tilde{T}^{k_1}(\Tilde{J}_1)$ is an interval of $\Tilde{X}$.

Thus every interval component $J$ of $X$ has a well-defined continuation $\Tilde{J} \subset \Tilde{X}$, for all sufficiently small perturbations. We now show that $\Tilde{X} = \bigsqcup_{\beta \in X \cap \, \mathcal{C}} \Tilde{O}(\Tilde{J}_\beta)$, which completes the proof. Let $X'$ be the union of all intervals $J'$ discussed above, i.e.\ the larger intervals containing the dynamically trivial and non-trivial intervals $J$ of $X$.

Since $\bigsqcup_{\beta \in X \cap \, \mathcal{C}} \Tilde{O}(\Tilde{J}_\beta)$ is compactly contained in the interior of $X'$, and every point in $X' \setminus \bigsqcup_{\beta \in X \cap \, \mathcal{C}} \Tilde{O}(\Tilde{J}_\beta)$ is eventually mapped into $\bigsqcup_{\beta \in X \cap \, \mathcal{C}} \Tilde{O}(\Tilde{J}_\beta)$, it is sufficient to show that all points in $I \setminus X'$ eventually get mapped into $X'$. Indeed, every such point in $I \setminus X'$ cannot be non-wandering. We first show this for the discontinuities in $I \setminus X'$, and then for every other point $z \in I \setminus X'$.

Assume first that for all sufficiently small perturbations, the itineraries of discontinuities in $I \setminus X'$ do not change up to their landing time into the interior of $X'$. Thus they have the same landing time into the interior of $X'$ for all sufficiently small perturbations. We may assume that the perturbation is sufficiently small, so that the itineraries of intervals of positive length containing the discontinuities in $I \setminus X'$ also do not change. We may assume these intervals are small enough so that they also land into the interior of $X'$ for all sufficiently small perturbations. Thus for all sufficiently small perturbations, none of these intervals is contained in $\Tilde{X}$. By the Orbit Classification Lemma \ref{lem:orb-class}, every point in $z \in I \setminus X'$ that is not eventually periodic either lands on a discontinuity or accumulates on a discontinuity. If $z$ lands on a discontinuity, it is clearly not in $\Tilde{X}$. If $z$ accumulates on a discontinuity in $X'$, we are done by the previous paragraph. If $z$ accumulates on a discontinuity outside of $X'$, it eventually lands into the interval of positive length containing this discontinuity that lands into the interior of $X'$, which means that it also eventually lands into $X'$. If $z \in I \setminus X'$ is eventually periodic, it lands into a periodic interval of $\Tilde{T}$. The boundary points of such an interval are in the orbits of discontinuities contained in $\Tilde{X}$. As we have shown that the only discontinuities in $\Tilde{X}$ are those contained in $X'$, $z$ therefore eventually lands into $\bigsqcup_{\beta \in X \cap \, \mathcal{C}} \Tilde{O}(\Tilde{J}_\beta)$. Thus every point eventually lands into $X'$, so the claim follows.

Now assume that for any sufficiently small perturbation, there is a discontinuity outside of $X'$ for which the itinerary changes before the landing time into the interior of $X'$. We may assume that the perturbation is small enough so that the itinerary of every discontinuity outside of $X'$ up to some finite time at which it does not land on another discontinuity remains the same as before the perturbation. Thus if $\beta \notin X'$ does not land on a discontinuity before the time it lands into the interior of $X'$, we may assume it still does so.

Let $\Tilde{T}$ be a sufficiently small perturbation as in the preceding paragraph. Let $\Tilde{\beta_{i_1}^+}$ (we may assume it is of $+$-type, with the other case being analogous) be a discontinuity that:
\begin{itemize}
    \item Changes itinerary before landing into the interior of $X'$;
    \item Does not eventually land into $X'$.
\end{itemize}
Such a discontinuity exists because if every discontinuity that changes itinerary still lands into $X'$, then there are again intervals of positive length containing such discontinuities that also land into $X'$, so a similar argument as before shows that every $z \in I \setminus X'$ eventually lands into $X'$.

We will show that there is a finite path $\beta_{i_1}^+ \leftarrow \beta_{i_2}^- \leftarrow \dots \leftarrow \betas$ in the ghost tree of $\beta_1^+$ such that $\betas$ already appears earlier in the path. This means that $\betas$ is contained in its ghost tree $\mathcal{GT}(\betas)$, which contradicts A3.

Let $n_1$ be the first time at which the itinerary of $\beta_{i_1}^+$ changes. Since the perturbation is small enough, $\beta_{i_1}^+$ must have landed on a discontinuity $\beta_{i_2}^+$ at time $n_1$, and now it lands to the left of it. We now claim that there is a discontinuity $\beta_{i_3}^-$ such that $\beta_{i_2}^-$ lands on $\beta_{i_3}^-$ at some time $n_2$ and such that $\Tilde{T}^{n_1 + n_2}(\Tilde{\beta}_{i_1}^+)$ lands to the right of $\Tilde{\beta}_{i_3}^+$.

Indeed, if $\beta_{i_2}^-$ does not change itinerary before landing into $X'$, the same holds for an interval of positive length to the left of $\Tilde{\beta}_{i_2}^-$ and thus for $\Tilde{T}^{n_1}(\Tilde{\beta}_{i_1}^+)$, which is a contradiction with our choice of $\beta_{i_1}^+$. Let $m_{2,1}$ be the first time at which $\beta_{i_2}^-$ changes itinerary. In the same way as above, $T^{m_{2,1}}(\beta_{i_2}^-) = \beta_{{i_2},1}^-$ for some discontinuity $\beta_{2,1}^-$ not contained in $X$ and that $\beta_{i_2}^-$ now lands to the right of $\beta_{i_2,1}^+$ after perturbation. If $\Tilde{T}^{n_1 + m_{2,1}}(\Tilde{\beta}_{i_1}^+)$ is to the left of $\Tilde{\beta}_{i_2,1}^-$, then we may repeat this argument for $\beta_{i_2,1}^-$, and get that it must land at some $\beta_{i_2,2}^-$ at time $m_{2,2}$, and it has to land to the right of $\Tilde{\beta}_{i_2,2}^+$ after the perturbation. Continuing by induction, we get that there is a minimal finite $t_2$ such that $\Tilde{T}^{n_1 + m_{2,1} + \dots + m_{2,t_2}}(\beta_{i_1}^+)$ lands to the right of $\beta_{i_2,t_2}^+$ and that we have $\beta_{i_2,1}^- \to \beta_{i_2,2}^- \to \dots \to \beta_{i_2,t_2}^-$. There has to exist such a finite $t_2$ because there are finitely many discontinuities, so each iterate $\Tilde{T}^{m_{2,i}}(\Tilde{\beta}_{i_2,i}^i)$ moves to the right by at least $\epsilon > 0$. Thus at each step $i$ of the induction $\beta_{i_2,i}^- - \Tilde{T}^{n_1 + m_{2,1} + \dots + m_{2,i}}(\beta_{i_1}^+)$ has to be smaller than $\beta_{i_2,i-1}^- - \Tilde{T}^{n_1 + m_{2,1} + \dots + m_{2,i-1}}(\beta_{i_1}^+)$ by at least $\epsilon$. We can therefore set $n_2 := m_{2,1} + \dots + m_{2,t_2}$ and $\beta_{i_3}^+ := \beta_{i_2,t_2}^+$.

This argument can now be repeated for $\beta_{i_3}^+$, so by induction we get an infinite path $\beta_{i_1}^+ \leftarrow \beta_{i_2}^- \leftarrow \beta_{i_3}^+ \leftarrow \dots$ in the ghost tree of $\beta_{i_1}^+$. Since there are only finitely many discontinuities, one of them has to be repeated, so we get the required finite sequence.
\end{proof}

\addcontentsline{toc}{section}{References}
\printbibliography

@book {MR1312365,
    AUTHOR = {McMullen, Curtis T.},
     TITLE = {Complex dynamics and renormalization},
    SERIES = {Annals of Mathematics Studies},
    VOLUME = {135},
 PUBLISHER = {Princeton University Press, Princeton, NJ},
      YEAR = {1994},
     PAGES = {x+214},
      ISBN = {0-691-02982-2; 0-691-02981-4},
   MRCLASS = {58F23 (30D05)},
  MRNUMBER = {1312365},
MRREVIEWER = {Gregery\ T.\ Buzzard},
}

@misc{drach2026topologicalprevalencefinitetype,
      title={Topological Prevalence of Finite Type Interval Translation Maps}, 
      author={Kostiantyn Drach and Leon Staresinic and Sebastian van Strien},
      year={2026},
      eprint={2605.00186},
      archivePrefix={arXiv},
      primaryClass={math.DS},
      url={https://arxiv.org/abs/2605.00186}, 
}

@misc{drach2026transversalityintervaltranslationmaps,
      title={Transversality for Interval Translation Maps}, 
      author={Kostiantyn Drach and Leon Staresinic and Sebastian van Strien},
      year={2026},
      eprint={2605.00173},
      archivePrefix={arXiv},
      primaryClass={math.DS},
      url={https://arxiv.org/abs/2605.00173}, 
}

@article {MR4973368,
    AUTHOR = {Artigiani, Mauro and Hubert, Pascal and Skripchenko,
              Alexandra},
     TITLE = {Renormalization for {B}ruin-{T}roubetzkoy {ITM}s},
   JOURNAL = {Discrete Contin. Dyn. Syst.},
  FJOURNAL = {Discrete and Continuous Dynamical Systems},
    VOLUME = {47},
      YEAR = {2026},
     PAGES = {519--547},
      ISSN = {1078-0947,1553-5231},
   MRCLASS = {37E05 (11K55 37A05 37A44 37C45)},
  MRNUMBER = {4973368},
       DOI = {10.3934/dcds.2025127},
       URL = {https://doi.org/10.3934/dcds.2025127},
}

@article {MR3893724,
    AUTHOR = {Bruin, Henk and Carminati, Carlo and Marmi, Stefano and
              Profeti, Alessandro},
     TITLE = {Matching in a family of piecewise affine maps},
   JOURNAL = {Nonlinearity},
  FJOURNAL = {Nonlinearity},
    VOLUME = {32},
      YEAR = {2019},
    NUMBER = {1},
     PAGES = {172--208},
      ISSN = {0951-7715,1361-6544},
   MRCLASS = {37E05 (11A55 11J70 11K50 37A45 37E45)},
  MRNUMBER = {3893724},
MRREVIEWER = {Steven\ M.\ Pederson},
       DOI = {10.1088/1361-6544/aae935},
       URL = {https://doi.org/10.1088/1361-6544/aae935},
}

@article {MR101951,
    AUTHOR = {Peixoto, M. M.},
     TITLE = {On structural stability},
   JOURNAL = {Ann. of Math. (2)},
  FJOURNAL = {Annals of Mathematics. Second Series},
    VOLUME = {69},
      YEAR = {1959},
     PAGES = {199--222},
      ISSN = {0003-486X},
   MRCLASS = {34.00},
  MRNUMBER = {101951},
MRREVIEWER = {L.\ Markus},
       DOI = {10.2307/1970100},
       URL = {https://doi.org/10.2307/1970100},
}

@article {MR339291,
    AUTHOR = {Newhouse, Sheldon E.},
     TITLE = {Diffeomorphisms with infinitely many sinks},
   JOURNAL = {Topology},
  FJOURNAL = {Topology. An International Journal of Mathematics},
    VOLUME = {13},
      YEAR = {1974},
     PAGES = {9--18},
      ISSN = {0040-9383},
   MRCLASS = {58F99},
  MRNUMBER = {339291},
MRREVIEWER = {Tudor\ S.\ Ra\c tiu},
       DOI = {10.1016/0040-9383(74)90034-2},
       URL = {https://doi.org/10.1016/0040-9383(74)90034-2},
}

@article {MR287580,
    AUTHOR = {Robbin, J. W.},
     TITLE = {A structural stability theorem},
   JOURNAL = {Ann. of Math. (2)},
  FJOURNAL = {Annals of Mathematics. Second Series},
    VOLUME = {94},
      YEAR = {1971},
     PAGES = {447--493},
      ISSN = {0003-486X},
   MRCLASS = {57.48},
  MRNUMBER = {287580},
MRREVIEWER = {Kenneth\ R.\ Meyer},
       DOI = {10.2307/1970766},
       URL = {https://doi.org/10.2307/1970766},
}

@article {MR474411,
    AUTHOR = {Robinson, Clark},
     TITLE = {Structural stability of {$C\sp{1}$} diffeomorphisms},
   JOURNAL = {J. Differential Equations},
  FJOURNAL = {Journal of Differential Equations},
    VOLUME = {22},
      YEAR = {1976},
    NUMBER = {1},
     PAGES = {28--73},
      ISSN = {0022-0396,1090-2732},
   MRCLASS = {58F10},
  MRNUMBER = {474411},
MRREVIEWER = {Zbigniew\ Nitecki},
       DOI = {10.1016/0022-0396(76)90004-8},
       URL = {https://doi.org/10.1016/0022-0396(76)90004-8},
}

@article {MR932138,
    AUTHOR = {Ma\~n\'e, Ricardo},
     TITLE = {A proof of the {$C^1$} stability conjecture},
   JOURNAL = {Inst. Hautes \'Etudes Sci. Publ. Math.},
  FJOURNAL = {Institut des Hautes \'Etudes Scientifiques. Publications
              Math\'ematiques},
    NUMBER = {66},
      YEAR = {1988},
     PAGES = {161--210},
      ISSN = {0073-8301,1618-1913},
   MRCLASS = {58F10 (58F15)},
  MRNUMBER = {932138},
MRREVIEWER = {Roberto\ Moriyon},
       URL =
              {http://www.numdam.org/numdam-bin/item?id=PMIHES_1987__66__161_0},
}

@article {MR2299743,
    AUTHOR = {Avila, Artur and Forni, Giovanni},
     TITLE = {Weak mixing for interval exchange transformations and
              translation flows},
   JOURNAL = {Ann. of Math. (2)},
  FJOURNAL = {Annals of Mathematics. Second Series},
    VOLUME = {165},
      YEAR = {2007},
    NUMBER = {2},
     PAGES = {637--664},
      ISSN = {0003-486X,1939-8980},
   MRCLASS = {37A25 (28D05 37D40 37E35)},
  MRNUMBER = {2299743},
MRREVIEWER = {Gabriela\ Schmith\"usen},
       DOI = {10.4007/annals.2007.165.637},
       URL = {https://doi.org/10.4007/annals.2007.165.637},
}

@misc{artigiani2026typicalweakmixingexceptional,
      title={Typical Weak Mixing and Exceptional Spectral Properties for Interval Translation Mappings}, 
      author={Mauro Artigiani and Artur Avila and Sébastien Ferenczi and Pascal Hubert and Alexandra Skripchenko},
      year={2026},
      eprint={2603.19401},
      archivePrefix={arXiv},
      primaryClass={math.DS},
      url={https://arxiv.org/abs/2603.19401}, 
}

@incollection {MR3821044,
    AUTHOR = {Berger, Pierre},
     TITLE = {Lectures on structural stability in dynamics},
 BOOKTITLE = {Dynamical systems},
    SERIES = {Banach Center Publ.},
    VOLUME = {115},
     PAGES = {9--35},
 PUBLISHER = {Polish Acad. Sci. Inst. Math., Warsaw},
      YEAR = {2018},
      ISBN = {978-83-86806-39-3},
   MRCLASS = {37C20 (37C70 37D20 37D50 37F15)},
  MRNUMBER = {3821044},
}

@article {MR142859,
    AUTHOR = {Peixoto, M. M.},
     TITLE = {Structural stability on two-dimensional manifolds},
   JOURNAL = {Topology},
  FJOURNAL = {Topology. An International Journal of Mathematics},
    VOLUME = {1},
      YEAR = {1962},
     PAGES = {101--120},
      ISSN = {0040-9383},
   MRCLASS = {57.48 (57.50)},
  MRNUMBER = {142859},
MRREVIEWER = {Jack\ K.\ Hale},
       DOI = {10.1016/0040-9383(65)90018-2},
       URL = {https://doi.org/10.1016/0040-9383(65)90018-2},
}

@article {MR3597033,
    AUTHOR = {Bruin, Henk and Carminati, Carlo and Kalle, Charlene},
     TITLE = {Matching for generalised {$\beta$}-transformations},
   JOURNAL = {Indag. Math. (N.S.)},
  FJOURNAL = {Koninklijke Nederlandse Akademie van Wetenschappen.
              Indagationes Mathematicae. New Series},
    VOLUME = {28},
      YEAR = {2017},
    NUMBER = {1},
     PAGES = {55--73},
      ISSN = {0019-3577,1872-6100},
   MRCLASS = {11K50 (37E05)},
  MRNUMBER = {3597033},
MRREVIEWER = {Alan\ Haynes},
       DOI = {10.1016/j.indag.2016.11.005},
       URL = {https://doi.org/10.1016/j.indag.2016.11.005},
}

@article {MR2422375,
    AUTHOR = {Nakada, Hitoshi and Natsui, Rie},
     TITLE = {The non-monotonicity of the entropy of {$\alpha$}-continued
              fraction transformations},
   JOURNAL = {Nonlinearity},
  FJOURNAL = {Nonlinearity},
    VOLUME = {21},
      YEAR = {2008},
    NUMBER = {6},
     PAGES = {1207--1225},
      ISSN = {0951-7715,1361-6544},
   MRCLASS = {37A45 (11K50 37A35)},
  MRNUMBER = {2422375},
MRREVIEWER = {Anne\ Broise-Alamichel},
       DOI = {10.1088/0951-7715/21/6/003},
       URL = {https://doi.org/10.1088/0951-7715/21/6/003},
}

@article {MR2342693,
    AUTHOR = {Kozlovski, Oleg and Shen, Weixiao and van Strien, Sebastian},
     TITLE = {Density of hyperbolicity in dimension one},
   JOURNAL = {Ann. of Math. (2)},
  FJOURNAL = {Annals of Mathematics. Second Series},
    VOLUME = {166},
      YEAR = {2007},
    NUMBER = {1},
     PAGES = {145--182},
      ISSN = {0003-486X,1939-8980},
   MRCLASS = {37E05 (37C20 37D20 37E10 37F30)},
  MRNUMBER = {2342693},
MRREVIEWER = {Mike\ Todd},
       DOI = {10.4007/annals.2007.166.145},
       URL = {https://doi.org/10.4007/annals.2007.166.145},
}

@article {MR2335796,
    AUTHOR = {Kozlovski, Oleg and Shen, Weixiao and van Strien, Sebastian},
     TITLE = {Rigidity for real polynomials},
   JOURNAL = {Ann. of Math. (2)},
  FJOURNAL = {Annals of Mathematics. Second Series},
    VOLUME = {165},
      YEAR = {2007},
    NUMBER = {3},
     PAGES = {749--841},
      ISSN = {0003-486X,1939-8980},
   MRCLASS = {37E05 (30C10 30C62 37C20 37E20 37F10 37F30 37F50)},
  MRNUMBER = {2335796},
MRREVIEWER = {Henk\ Bruin},
       DOI = {10.4007/annals.2007.165.749},
       URL = {https://doi.org/10.4007/annals.2007.165.749},
}

@article {MR1620850,
    AUTHOR = {McMullen, Curtis T. and Sullivan, Dennis P.},
     TITLE = {Quasiconformal homeomorphisms and dynamics. {III}. {T}he
              {T}eichm\"uller space of a holomorphic dynamical system},
   JOURNAL = {Adv. Math.},
  FJOURNAL = {Advances in Mathematics},
    VOLUME = {135},
      YEAR = {1998},
    NUMBER = {2},
     PAGES = {351--395},
      ISSN = {0001-8708,1090-2082},
   MRCLASS = {58F23 (30C62 30D05 30F60)},
  MRNUMBER = {1620850},
MRREVIEWER = {Athanase\ Papadopoulos},
       DOI = {10.1006/aima.1998.1726},
       URL = {https://doi.org/10.1006/aima.1998.1726},
}

@book {MR1239171,
    AUTHOR = {de Melo, Welington and van Strien, Sebastian},
     TITLE = {One-dimensional dynamics},
    SERIES = {Ergebnisse der Mathematik und ihrer Grenzgebiete (3) [Results
              in Mathematics and Related Areas (3)]},
    VOLUME = {25},
 PUBLISHER = {Springer-Verlag, Berlin},
      YEAR = {1993},
     PAGES = {xiv+605},
      ISBN = {3-540-56412-8},
   MRCLASS = {58F03 (58-02 58Fxx)},
  MRNUMBER = {1239171},
MRREVIEWER = {Feliks\ Przytycki},
       DOI = {10.1007/978-3-642-78043-1},
       URL = {https://doi.org/10.1007/978-3-642-78043-1},
}

@article{bruin2023interval,
  title={Interval Translation Maps with Weakly Mixing Attractors},
  author={Bruin, Henk and Radinger, Silvia},
  journal={arXiv preprint arXiv:2312.10533},
  year={2023}
}

@article{MR4397159,
	author = {Artigiani, Mauro and Fougeron, Charles and Hubert, Pascal and Skripchenko, Alexandra},
	doi = {10.1090/mosc/311},
	fjournal = {Transactions of the Moscow Mathematical Society},
	issn = {0077-1554,1547-738X},
	journal = {Trans. Moscow Math. Soc.},
	mrclass = {37E05},
	mrnumber = {4397159},
	mrreviewer = {Jo\~ao\ Lopes Dias},
	pages = {157--172},
	title = {A note on double rotations of infinite type},
	url = {https://doi.org/10.1090/mosc/311},
	volume = {82},
	year = {2021}}

@article{MR2966738,
	author = {Bruin, Henk and Clack, Gregory},
	doi = {10.3934/dcds.2012.32.4133},
	fjournal = {Discrete and Continuous Dynamical Systems. Series A},
	issn = {1078-0947,1553-5231},
	journal = {Discrete Contin. Dyn. Syst.},
	mrclass = {37B10 (37C70 37D25 37E05 37E10)},
	mrnumber = {2966738},
	mrreviewer = {David\ Ralston},
	number = {12},
	pages = {4133--4147},
	title = {Inducing and unique ergodicity of double rotations},
	url = {https://doi.org/10.3934/dcds.2012.32.4133},
	volume = {32},
	year = {2012}}

@article{MR2308208,
	author = {Bruin, Henk},
	doi = {10.1080/14689360601028084},
	fjournal = {Dynamical Systems. An International Journal},
	issn = {1468-9367,1468-9375},
	journal = {Dyn. Syst.},
	mrclass = {37E20 (37E05 37E10)},
	mrnumber = {2308208},
	mrreviewer = {J\'er\^ome\ Buzzi},
	number = {1},
	pages = {11--24},
	title = {Renormalization in a class of interval translation maps of {$d$} branches},
	url = {https://doi.org/10.1080/14689360601028084},
	volume = {22},
	year = {2007}}

@article{MR2152403,
	author = {Suzuki, Hideyuki and Ito, Shunji and Aihara, Kazuyuki},
	doi = {10.3934/dcds.2005.13.515},
	fjournal = {Discrete and Continuous Dynamical Systems. Series A},
	issn = {1078-0947,1553-5231},
	journal = {Discrete Contin. Dyn. Syst.},
	mrclass = {37E05 (37B05 37E10 37E45)},
	mrnumber = {2152403},
	mrreviewer = {Henk\ Bruin},
	number = {2},
	pages = {515--532},
	title = {Double rotations},
	url = {https://doi.org/10.3934/dcds.2005.13.515},
	volume = {13},
	year = {2005}}

@article{MR2013352,
	author = {Bruin, Henk and Troubetzkoy, Serge},
	doi = {10.1007/BF02785958},
	fjournal = {Israel Journal of Mathematics},
	issn = {0021-2172,1565-8511},
	journal = {Israel J. Math.},
	mrclass = {37A25 (37B99 37E10)},
	mrnumber = {2013352},
	mrreviewer = {Franco\ Vivaldi},
	pages = {125--148},
	title = {The {G}auss map on a class of interval translation mappings},
	url = {https://doi.org/10.1007/BF02785958},
	volume = {137},
	year = {2003}}

@article{MR1356616,
	author = {Boshernitzan, Michael and Kornfeld, Isaac},
	doi = {10.1017/S0143385700009652},
	fjournal = {Ergodic Theory and Dynamical Systems},
	issn = {0143-3857,1469-4417},
	journal = {Ergodic Theory Dynam. Systems},
	mrclass = {58F03 (58F11)},
	mrnumber = {1356616},
	mrreviewer = {Jos\'e\ Miguel\ Moreno},
	number = {5},
	pages = {821--832},
	title = {Interval translation mappings},
	url = {https://doi.org/10.1017/S0143385700009652},
	volume = {15},
	year = {1995}}

@article{MR3124735,
	author = {Volk, Denis},
	doi = {10.3934/dcds.2014.34.2307},
	fjournal = {Discrete and Continuous Dynamical Systems. Series A},
	issn = {1078-0947,1553-5231},
	journal = {Discrete Contin. Dyn. Syst.},
	mrclass = {37C70 (37C05 37C20 37D20 37D45)},
	mrnumber = {3124735},
	number = {5},
	pages = {2307--2314},
	title = {Almost every interval translation map of three intervals is finite type},
	url = {https://doi.org/10.3934/dcds.2014.34.2307},
	volume = {34},
	year = {2014}}

@article {MR644019,
    AUTHOR = {Veech, William A.},
     TITLE = {Gauss measures for transformations on the space of interval
              exchange maps},
   JOURNAL = {Ann. of Math. (2)},
  FJOURNAL = {Annals of Mathematics. Second Series},
    VOLUME = {115},
      YEAR = {1982},
    NUMBER = {1},
     PAGES = {201--242},
      ISSN = {0003-486X},
   MRCLASS = {28D05 (58F11)},
  MRNUMBER = {644019},
MRREVIEWER = {Michael\ Keane},
       DOI = {10.2307/1971391},
       URL = {https://doi.org/10.2307/1971391},
}

@article {MR644018,
    AUTHOR = {Masur, Howard},
     TITLE = {Interval exchange transformations and measured foliations},
   JOURNAL = {Ann. of Math. (2)},
  FJOURNAL = {Annals of Mathematics. Second Series},
    VOLUME = {115},
      YEAR = {1982},
    NUMBER = {1},
     PAGES = {169--200},
      ISSN = {0003-486X},
   MRCLASS = {28D05 (10K10 58F11 58F18)},
  MRNUMBER = {644018},
MRREVIEWER = {D.\ Newton},
       DOI = {10.2307/1971341},
       URL = {https://doi.org/10.2307/1971341},
}

@misc{drach2025densitystableintervaltranslation,
      title={Density of Stable Interval Translation Maps}, 
      author={Kostiantyn Drach and Leon Staresinic and Sebastian van Strien},
      year={2025, \url{https://arxiv.org/abs/2411.14312}},
      eprint={2411.14312},
      archivePrefix={arXiv},
      primaryClass={math.DS},
      url={https://arxiv.org/abs/2411.14312}, 
}

\end{document}